\documentclass[a4paper]{article}

\usepackage[utf8]{inputenc}
\usepackage[T1]{fontenc}
\usepackage{lmodern}
\usepackage{cite}

\usepackage{hyperref} 
\usepackage{amssymb,amsmath,amsthm}
\usepackage[english]{babel}
\usepackage{todonotes}
\usepackage{verbatim} 
\usepackage{enumerate}
\usepackage{mathtools}
\usepackage{bbm}
\usepackage[normalem]{ulem}
\usepackage{tikz}
\usetikzlibrary{matrix}

\newcommand{\cA}{\mathcal{A}}

\newcommand{\cC}{\mathcal{C}}
\newcommand{\cD}{\mathcal{D}}

\newcommand{\cH}{\mathcal{H}}
\newcommand{\cI}{\mathcal{I}}
\newcommand{\cJ}{\mathcal{J}}

\newcommand{\cL}{\mathcal{L}}

\newcommand{\cQ}{\mathcal{Q}}
\newcommand{\cR}{\mathcal{R}}

\newcommand{\cT}{\mathcal{T}}
\newcommand{\cU}{\mathcal{U}}

\newcommand{\cX}{\mathcal{X}}
\newcommand{\cY}{\mathcal{Y}}

\newcommand{\bR}{\mathbb{R}}

\newcommand{\bONE}{\mathbbm{1}}
\newcommand{\Ba}{Ba}

\newcommand{\dd}{ \mathrm{d}}

\DeclareMathOperator*{\LIM}{LIM} 
\DeclareMathOperator*{\subLIM}{subLIM}
\DeclareMathOperator*{\superLIM}{superLIM}

\renewcommand{\epsilon}{\varepsilon}

\newcommand{\vn}[1]{\left| \! \left| #1\right| \! \right|}

\newcommand{\ip}[2]{\langle #1,#2\rangle}

\numberwithin{equation}{section}

\newtheorem{theorem}{Theorem}[section]
\newtheorem{lemma}[theorem]{Lemma}
\newtheorem{proposition}[theorem]{Proposition}

\theoremstyle{definition}
\newtheorem{definition}[theorem]{Definition}

\newtheorem{remark}[theorem]{Remark}

\newtheorem{assumption}[theorem]{Assumption}
\newtheorem{example}[theorem]{Example}
\newtheorem{condition}[theorem]{Condition}

\begin{document}

\title{Gamma convergence on path-spaces via convergence of viscosity solutions of Hamilton-Jacobi equations}

\author{Richard C. Kraaij\thanks{Delft Institute of Applied Mathematics, Delft University of Technology, Van Mourik Broekmanweg 6, 2628 XE Delft, The Netherlands. \emph{E-mail address}: r.c.kraaij@tudelft.nl}}
\date{\today}

	\maketitle

	\begin{abstract}
			We establish a framework that allows to prove Gamma-converge of functionals of Lagrangian form on spaces of trajectories based on convergence of viscosity solutions of associated Hamilton-Jacobi equations. 
			
			Gamma convergence follows from a: equi-coercivity, b: Gamma convergence of the projected functional at time $0$, c: convergence of the Hamiltonians that appear as Legendre transform of the Lagrangian in the path-space functional.

		\smallskip


	\end{abstract}


\section{Introduction}

When considering Gamma convergence, see \cite{Br02,DM93}, of functionals $F_n$ to $F$, it is well known that helps to express the functionals $F^n$ and $F$ in terms of an operator for which one can prove convergence. 

Consider for example a Hilbert space $H$ equipped with its weak topology and self adjoint operators $A_n,A$ satisfying $\ip{Ah}{h} \geq \lambda \vn{h}^2$ and $\ip{Ah}{h} \geq \lambda \vn{h}^2$. Then Gamma convergence of $F^n$ to $F$ where
\begin{equation*}
F^n(h) := \sup_{h_0 \in \cD(A_n)} \ip{A_n h_0}{2h-h_0}, \qquad F(h) := \sup_{h_0 \in \cD(A)} \ip{A h_0}{2h-h_0},
\end{equation*}
is equivalent to G-convergence, a notion related to convergence of the resolvents of $A$. See Chapter 13 of \cite{Br02}.

\smallskip

Instead of considering Hilbert spaces, we will consider the context of trajectories on some space $X$ (thus, e.g. $C(X)$) and functionals $I^n,I : C(X) \rightarrow [0,\infty]$ of the type
\begin{equation} \label{eqn:functional_intro}
I^n(\gamma) = \begin{cases}
I_0^n(\gamma(0)) + \int_0^\infty \cL_n(\gamma(s),\dot{\gamma}(s)) \dd s & \text{if } \gamma \in \cA\cC(X), \\
\infty & \text{otherwise},
\end{cases}
\end{equation}
where $\cA\cC(X)$ is some space of absolutely continuous trajectories on $X$. $\cL_n : TX \rightarrow [0,\infty]$ is a `Lagrangian' that could be some quadratic norm of the speed $\dot{\gamma}(s)$ or a functional related to the entropy necessary to create the `speed' $\dot{\gamma}(s)$ at time $s$.

\smallskip

In this paper, we establish that the operator to look at in this context to obtain Gamma convergence of $I^n$ to $I$ are related to the Hamiltonian $\cH_n : T^* X \rightarrow \bR$ obtained by taking the Legendre transforms of $\cL_n$. If one formally defines $H_n f(x) := \cH_n(x,\dd f(x))$ then our main result essentially states: 
\begin{itemize}
	\item Suppose that the functionals $I^n$ are equi-coercive;
	\item Suppose that the functionals $I^n_0$ Gamma converge to $I_0$;
	\item Suppose that for each $f \in \cD(H)$ there are $f_n \in \cD(H_n)$ such that $f_n \rightarrow f$ and $H_nf_n \rightarrow f$;
	\item The Hamilton-Jacobi equation $f - \lambda Hf = h$ is well posed for all $\lambda > 0$ and $h \in C_b(X)$.
\end{itemize}
Then we have Gamma convergence of $I^n$ to $I$.

\smallskip

Our result is stated for functionals $I^n,I$ on a space of the type $C(X)$ with the compact-open topology. This setting introduces the need, and the possibility, to use different techniques than on Hilbert spaces. One of the key advantages of $C(X)$ is its projective limit structure. Namely, if $\gamma_n,\gamma$ are in some compact set $K \subseteq C(X)$, than $\gamma_n \rightarrow \gamma$ for the compact-open topology if and only if $\gamma_n(t) \rightarrow \gamma(t)$ for all times $t \geq 0$.

As a consequence of this fact, the functionals $I^n,I$ can be written in a projective limit form:
\begin{align*}
I^n(\gamma) & = \sup_{\{0 = t_0 < t_1 < \dots < t_k\}} I^n[t_0,t_1,\dots,t_k](\gamma(t_0),\dots,\gamma(t_k)) \\
& = \sup_{\{0=t_0 < t_1 < \dots < t_k\}} I_0(\gamma(0)) + \sum_{i=1}^k I_{t_i - t_{i-1}}(\gamma(t_i) \, | \, \gamma(t_{i-1})),
\end{align*}
where
\begin{equation*}
I_t(y \, | \, x) = \inf_{\substack{\gamma \in \cA\cC \\ \gamma(0) = x, \gamma(t) = y}} \int_0^t \cL_n(\gamma(s),\dot{\gamma}(s)) \dd s.
\end{equation*}
Morally, given equi-coercivity of the sequence $I^n$, this allows one to reduce the analysis of $I^n$ and $I$ to that of $I^n_0, I_0$ and the conditional functionals $I_t^n(\cdot \, | \, \cdot)$ and $I_t(\cdot \, | \, \cdot)$. Such a procedure is well known in the context of weak convergence and large deviations of Markov processes and we refer to \cite{EK86,FK06} for two accounts of these two topics. In both these contexts, the convergence or large deviations of the finite dimensional distributions is reduced to the convergence of semigroups and afterwards to that of `generators'. As weak convergence and large deviations are special case of Gamma convergence, see \cite{Ma18}, a similar method of proof can be expected to work in the context of Gamma convergence and this is what this paper establishes.

\smallskip

In the context of Gamma convergence, the conditional functionals can indeed also be rewritten in terms of a semigroup. Namely
\begin{equation*}
I_t(y \, | \, x) = \sup_{f \in C_b(X) } \left\{f(y) - V(t)f(x) \right\}
\end{equation*}
where the semigroup $V(t)$ is of the form
\begin{equation*}
V(t)f(x) =  \sup_{\substack{\gamma \in \cA\cC \\ \gamma(0) = x}} \left\{f(\gamma(t)) - \int_0^t \cL_n(\gamma(s),\dot{\gamma}(s)) \dd s \right\}.
\end{equation*}
A formal derivative of $V(t)f$ yields that its generator is indeed $Hf(x) = \cH(x,\nabla f(x))$, where $\cH$ is the Legendre transform of $\cL$.

Thus, Gamma convergence of $I_0^n$ to $I_0$ and convergence of the Hamiltonians $H_n$ constructed from $\cL_n$ to the Hamiltonian $H$ constructed from $\cL$ should yield Gamma convergence of $I^n$ to $I$.

Thus, a goal is to obtain convergence of the semigroups $V_n(t)$ to $V(t)$, arguing via the Hamiltonians $H_n$ and $H$. Convergence results in the context of non-linear semigroups are technically challenging. One can use e.g. a combination of Crandall-Liggett \cite{CL71} and the Kurtz approximation procedure \cite{Ku73} to establish that if generators converge, then their semigroups convergence. However, one needs to establish the range condition: i.e. well posedness of the equation $f - \lambda Hf = h$ in the classical sense. For non-linear equations, this is difficult and it was observed early on \cite{CL83} that viscosity solutions can be used to replace classical solutions. Thus instead of focussing on convergence of $H_n$ to $H$, we will focus on establishing that viscosity solutions to $f - \lambda H_n f =h $ converge to viscosity solutions of $f - \lambda Hf = h$. To do so, we work via a generalization \cite{Kr19} of the semi-relaxed $\limsup$ and $\liminf$ procedure introduced by \cite{BaPe88,BaPe90} for which we also need uniqueness of viscosity solutions to the Hamilton-Jacobi equation $f - \lambda Hf = h$ and some notion of control of the viscosity solutions and semigroups as the space $X$ might be non-compact.

\smallskip

The paper is organized as follows. In the next Section, \ref{section:basic_main}, we state a simplified version of our main result (\ref{theorem:main_theorem_basic}). In this version, technicalities have been reduced to the bare minimum. We also give the main definitions on Gamma convergence, equi-coercivity and convergence of Hamiltonians.

In Section \ref{section:preliminaries}, we proceed with a more general set-up. This more general set-up allows for the extension of our main result to achieve the following:
\begin{itemize}
	\item We can work with spaces of trajectories that allow for discontinuities. The key property allowing such a generalization is that the topology is in some sense determined by the finite dimensional projections.
	\item We can consider functionals $I^n$ on spaces of paths that take their values on a sequence of different spaces. This will allow for embeddings or settings with e.g. homogenisation.
	\item We will work with pairs of Hamiltonians $H_{n,\dagger},H_{n,\ddagger}$ that serve as an upper and lower bound for $H_n$. This is natural for Hamilton-Jacobi equations in infinite dimensional context in which it is hard to explicitly give $H_n$. See e.g. \cite{Fe06,FeMiZi18,AmFe14,CrLi94}.
\end{itemize}

We give our extended main result in Section \ref{section:main_results} and the proofs follow in Section \ref{section:proofs_main_results}.


\section{A basic result on Gamma convergence} \label{section:basic_main}

Before going to the most general version of the main result, we introduce a basic variant in which versatility of the main result is reduced to obtain a easy to understand statement.

In Section \ref{subsection:basic_main_statement} we state this result without introducing various definitions. Most essential definitions on Gamma-convergence, functionals and how they are determined by semigroups are stated immediately afterwards in Section \ref{subsection:basic_main_definitions}. What are viscosity solutions, Hamiltonians and pseudo-resolvents  and how pseudo-resolvents generate semigroups can be found in slightly more general context in Appendix \ref{appendix:operators_topology_and_viscosity_solutions}.

\smallskip

First some general definitions. All spaces in this paper are assumed to be completely regular spaces that have metrizable compact sets. Let $X$ be a space then we denote by $C_b(X)$ the set of continuous and bounded functions into $\bR$. We denote by $\Ba(X)$ the space of Baire measurable sets (the $\sigma$-algebra generated by $C_b(X)$.) By $M(X)$, we denote by set of Baire measurable functions from $X$ into $\overline{\bR} := [-\infty,\infty]$. $M_b(X)$ denotes the set of bounded Baire measurable functions.  Denote
\begin{align*}
USC_u(X) & := \left\{f \in M(X) \, \middle| \,  f \, \text{upper semi-continuous}, \sup_x f(x) < \infty \right\}, \\
LSC_l(X) & := \left\{f \in M(X) \, \middle| \, f \, \text{lower semi-continuous}, \inf_x f(x) > \infty \right\}.
\end{align*}
For $g \in M(X)$ denote by $g^*,g_* \in M(X)$ the upper and lower semi-continuous regularizations of $g$.

Finally, we denote by $C_X(\bR^+)$ the space of continuous trajectories $\gamma : \bR^+ \rightarrow X$ and by $D_X(\bR^+)$, cf. \cite{EK86}, the Skorokhod space of trajectories that are right continuous and have left limits.

\subsection{Gamma convergence via convergence of Hamiltonians} \label{subsection:basic_main_statement}

\begin{condition} \label{condition:convergence_of_generators_basic}
	There are operators $H_n \subseteq C_b(X) \times C_b(X)$, contractive pseudo-resolvents $R_n(\lambda) : C_b(X) \rightarrow C_b(X)$, $\lambda >0$ and contractive semigroups $V_n(t) : C_b(X) \rightarrow C_b(X)$ generated by $R_n(\lambda)$.	These operators have the following properties:
	\begin{enumerate}[(a)]
		\item \label{item:convH_pseudoresolvents_solve_HJ_basic} For each $n \geq 1$, $\lambda > 0$ and $h \in C_b(X)$ the function $R_n(\lambda)h$ is a viscosity solution to $f - \lambda  H_n f= h$. 
		\item \label{item:convH_strict_equicont_resolvents_basic} We have local strict equi-continuity on bounded sets for the resolvents: for all compact sets $K \subseteq X$, $\delta > 0$ and $\lambda_0 > 0$, there is a compact set $\hat{K} = \hat{K}(K,\delta,\lambda_0)$ such that for all $n$ and $h_{1},h_{2} \in C_b(X)$ and $0 < \lambda \leq \lambda_0$ that
		\begin{multline*}
		\sup_{y \in K} \left\{ R_n(\lambda)h_{1}(y) - R_n(\lambda)h_{2}(y) \right\} \\
		\leq \delta \sup_{x \in X} \left\{ h_{1}(x) - h_{2}(x) \right\} + \sup_{y \in \hat{K}} \left\{ h_{1}(y) - h_{2}(y) \right\}.
		\end{multline*}
		\item \label{item:convH_strict_equicont_semigroups_basic}  We have local strict equi-continuity on bounded sets for the semigroups: for all compact sets $K \subseteq X$, $\delta > 0$ and $t_0 > 0$, there is a compact set $\hat{K} = \hat{K}(K,\delta,\lambda_0)$ such that for all $n$ and $h_{1},h_{2} \in C_b(X)$ and $0 \leq t \leq t_0$ that
		\begin{multline*}
		\sup_{y \in K} \left\{ V_n(t)h_{1}(y) - V_n(t)h_{2}(y) \right\} \\
		\leq \delta \sup_{x \in X} \left\{ h_{1}(x) - h_{2}(x) \right\} + \sup_{y \in \hat{K}} \left\{ h_{1}(y) - h_{2}(y) \right\}.
		\end{multline*}
	\end{enumerate}
\end{condition}

\begin{theorem} \label{theorem:main_theorem_basic}
	Suppose that we have coercive path-space functionals $I^n : C_X(\bR^+) \rightarrow [0,\infty]$ and $I : C_X(\bR^+) \rightarrow [0,\infty]$ that are generated by semigroups $V_n(t)$ and $V(t)$ with initial functionals $I_0^n$ and $I_0$.
	\begin{enumerate}[(a)]
		\item The functionals $I^n$ are equi-coercive,
		\item The functionals $I^n_0$ Gamma-converge to $I_0$
		\item Condition \ref{condition:convergence_of_generators_basic} is satisfied.
		\item We have $H \subseteq \LIM H_n$;
		\item The comparison principle holds for the Hamilton-Jacobi equations $f - \lambda H f = h$ for all $h \in C_b(X)$ and $\lambda > 0$. Denote the unique solution by $R(\lambda)h$.
		\item For all $h \in C_b(X)$ and $\lambda > 0$ It holds that $R(\lambda)h$ buc converges to $h$ (bounded and uniform on compacts).
	\end{enumerate}
	Then we have $I = \Gamma-\lim I^n$.
\end{theorem}

This result will follow as a special case of a much more general result that we will prove in the following sections. Generalizations include various steps
\begin{itemize}
	\item We can consider functionals $I^n : C_{X_n}(\bR^+) \rightarrow [0,\infty]$ on a sequence of spaces $X_n$ that are mapped into $X$ by some maps $\eta_n : X_n \rightarrow X$.
	\item Instead of considering the space of continuous functions on $X_n$ and $X$, we can consider more general path-spaces like the Skorokhod space.
	\item We will work with upper and lower bounds $H_{n,\dagger}$ and $H_{n,\ddagger}$ for $H_n$, as well as a natural upper and lower bound $H_\dagger$ and $H_\ddagger$ for $H$. This way, we can also relax $H \subseteq ex-\LIM H_n$ by letting $H_\dagger$ be a asymptotic upper bound for $H_{n,\dagger}$ and $H_\ddagger$ a asymptotic upper bound.
\end{itemize}

\subsection{Basic definitions} \label{subsection:basic_main_definitions}

We state the basic definitions on coercivity and gamma convergence. Afterwards we give the definition of a path-space functional determined by a semigroup.

\begin{definition} \label{definition:equi_coercive_basic}
	Let $X$ be some space and let $I : X \rightarrow [0,\infty]$ and $I^n : X \rightarrow [0,\infty]$. We say that
	\begin{itemize}
		\item $I$ is \textit{lower semi-continuous} if for each $M \geq 0$, the set $\{x \in X \, | \, I(x) \leq M\}$ is closed. 
		\item $I$ is \textit{coercive} if for every $M$ the set $\{x \in X \, | \, I(x) \leq M\}$ is compact.
		\item the sequence of lower semi-continuous functionals $I^n : X \rightarrow [0,\infty]$ is \textit{equi-coercive} if for every $M \geq 0$ there exists a compact set such that for all $n \geq 1$, we have
		\begin{equation*}
		\left\{x \in X \, \middle| \, I^n(x) \leq M\right\} \subseteq  K.
		\end{equation*}
	\end{itemize}
\end{definition}

\begin{definition}
	Let $I^n : X \rightarrow [0,\infty]$ and $I : X \rightarrow [0,\infty]$ be functionals.
	\begin{enumerate}[(a)]
		\item We say that the Gamma convergence lower bound holds, i.e. $I \leq \Gamma-\liminf_n I^n$ if for all $x_n$ such that $x_n  \rightarrow x$ it holds that
		\begin{equation*}
		I(x) \leq  \liminf_n I^n(x_n).
		\end{equation*}
		\item We say that the Gamma convergence upper bound holds, i.e. $I \geq \Gamma-\limsup_n I^n$ if for each $x$ there are $x_n$ such that $x_n \rightarrow x$ and
		\begin{equation*}
		I(x) \geq \limsup_n I^n(x_n).
		\end{equation*}
		\item We say that $I^n$ Gamma converges to $I$, i.e. $I = \Gamma-\lim I^n$ if the upper and lower bound hold.
	\end{enumerate}
\end{definition}

Dual to the notion of gamma convergence there is the notion of bounded and uniform convergence on compacts.

\begin{definition}
	Let $f_n, f \in C_b(X)$ we say that $\LIM_n f_n = f$ if $\sup_n \vn{f_n} < \infty$ and if $f_n$ converges to $f$ uniformly on compacts.
	
	For operators $H_n \subseteq C_b(X) \times C_b(X)$ and $H \subseteq C_b(X) \times C_b(X)$, we say that $H \subseteq ex-\LIM H_n$ if for all $(f,g) \in H$ there are $(f_n,g_n) \in H_n$ such that $\LIM f_n = f$ and $\LIM g_n = g$.
\end{definition}

\begin{definition}[Finite-dimensional determination of path-space functional]  \label{definition:representation_rate_function1_basic}
	We say that $I : C_X(\bR^+) \rightarrow [0,\infty]$ is a \textit{path-space functional} if it has projective limit form:
	\begin{equation} \label{eqn:functional_project}
	I(\gamma)  = \sup_{k \geq 0} \quad \sup_{0 = t_0 < t_1 < \dots < t_k} I[t_0,\dots,t_k](\gamma(t_0),\dots,\gamma(t_k)), 
	\end{equation}
	where
	\begin{equation}\label{eqn:functional_contract}
	I[t_0,\dots,t_k](x_0,\dots,x_k) = \inf \left\{I(\gamma) \, \middle| \, \gamma \in C_X(\bR^+), \forall i \in \{0,\dots,k\}: \gamma(t_i) = x_i\right\}.
	\end{equation}
\end{definition}

\begin{definition}[Path space functional determined by a semigroup]  \label{definition:representation_rate_function2_basic}
	Let $I :C_X(\bR^+) \rightarrow [0,\infty]$ be a path-space functional. Let $V(t)$ be a non-linear semigroup on $X$, i.e. $V(t) : C_b(X) \rightarrow C_b(X)$, where $V(0) = \bONE$, and $V(s)V(t) = V(t+s)$. We say that $I$ is determined by the semigroup $V$ and initial functional $I_0$ if
	\begin{equation} \label{eqn:functional_finite_d_Markov}
	I[t_0,\dots,t_k](x_0,\dots,x_k) = I_0(x_0) + \sum_{i=1}^k I_{t_i - t_{i-1}}(x_i \, | \, x_{i-1}).
	\end{equation}
	where
	\begin{equation*}
	I_t(y \, | \, x) := \sup_{f \in C_b(X)} \left\{f(y) - V(t)f(x) \right\}.
	\end{equation*}
\end{definition}

\begin{example} \label{example:Lagrangian_path_space_functionals}
	As mentioned in the introduction, the main functionals of interest are of the following type:
	\begin{equation*}
	I(\gamma) = \begin{cases}
	I_0(\gamma(0)) + \int_0^\infty \cL(\gamma(s),\dot{\gamma}(s)) \dd s & \text{if } \gamma \in \cA\cC(X), \\
	\infty & \text{otherwise},
	\end{cases}
	\end{equation*}
	Here $I_0 : \bR^d \rightarrow [0,\infty]$ is coercive and let $\cL : \bR^d \times \bR^d \rightarrow [0,\infty]$ is a Lagrangian, i.e. $v \mapsto \cL(x,v)$ is convex and $(x,v) \mapsto \cL(x,v)$ is lower semi-continuous. $\cA\cC(X)$ is some space absolutely continuous trajectories.
	
	The corresponding semigroup and resolvent are given by
	\begin{align*}
	V(t)f(x) & =  \sup_{\substack{\gamma \in \cA\cC \\ \gamma(0) = x}} \left\{f(\gamma(t)) - \int_0^t \cL(\gamma(s),\dot{\gamma}(s)) \dd s \right\}, \\
	R(\lambda)f(x) & =  \sup_{\substack{\gamma \in \cA\cC \\ \gamma(0) = x}} \left\{\int_0^\infty \lambda^{-1} e^{- \lambda^{-1} t} \dd t \left[f(\gamma(t)) - \int_0^t \cL(\gamma(s),\dot{\gamma}(s)) \dd s \right] \right\}.
	\end{align*}
	For a thorough analysis that connects viscosity solutions, resolvents, semigroups and Hamiltonians, see \cite[Chapter 8]{FK06}. Note that establishing equi-coercivity in this setting requires similar techniques to that of controlling the resolvent and semigroup.
\end{example}

\section{Preliminaries} \label{section:preliminaries}

We proceed with the preliminaries necessary to state the generalization of Theorem \ref{theorem:main_theorem_basic}. As an important ingredient, we will introduce `converging' sequence of spaces. In the sections below, we will adjust the notions of equi-coercivity, Gamma convergence, $\LIM$ and that of path-spaces to the more general context.

We start by introducing the notion of a converging sequence of spaces.

\subsection{A converging sequence of spaces}

\begin{definition}[Kuratowski convergence]
	Let $\{O_n\}_{n \geq 1}$ be a sequence of subsets in a space $X$. We define the \textit{limit superior} and \textit{limit inferior} of the sequence as
	\begin{align*}
	\limsup_{n \rightarrow \infty} O_n & := \left\{x \in X \, \middle| \, \forall \, U \in \cU_x \, \forall \, N \geq 1 \, \exists \, n \geq N: \, O_n \cap U \neq \emptyset \right\}, \\
	\liminf_{n \rightarrow \infty} O_n & := \left\{x \in X \, \middle| \, \forall \, U \in \cU_x \, \exists \, N \geq 1 \, \forall \, n \geq N: \, O_n \cap U \neq \emptyset \right\}.
	\end{align*}
	where $\cU_x$ is the collection of open neighbourhoods of $x$ in $X$. If $O := \limsup_n O_n = \liminf_n O_n$, we write $O = \lim_n O_n$ and say that $O$ is the Kuratowski limit of the sequence $\{O_n\}_{n \geq 1}$.
\end{definition}

\begin{assumption} \label{assumption:abstract_spaces_q}
	Consider spaces $X_n$ and $X$ and continuous maps $\eta_n : X_n \rightarrow X$. There is a directed set $\cQ$ (partially ordered set such that every two elements have an upper bound). For each $q \in \cQ$, we have compact sets $K_n^q \subseteq X_n$ a compact set $K^q \subseteq X$ such that
	\begin{enumerate}[(a)]
		\item If $q_1 \leq q_2$, we have $K^{q_1} \subseteq K^{q_2}$ and for all $n$ we have $K_{n}^{q_1} \subseteq K_n^{q_2}$.
		\item For all $q \in \cQ$ we have $\bigcup_n \eta_n(K_n^q) \subseteq K^q$.
		\item For each compact set $K \subseteq X$, there is a $q \in \cQ$ such that
		\begin{equation*}
		K \subseteq \liminf_n \eta_n(K_n^q).
		\end{equation*}
	\end{enumerate}
\end{assumption}

\begin{remark}
	The author expects (b) can be relaxed to a statement related to $\limsup \eta_n(K_n^q)$:
	
	\textit{For all $q \in \cQ$ and each sequence $x_n \in K_n^q$, every subsequence of $x_n$ has a further subsequence that is converging to a limit $x \in K^q$ (that is: $\eta_n(x_n) \rightarrow x$ in $X$).}
	
	This version of assumption, leads to problems in the final estimate of Proposition \ref{proposition:conditional_upper_bound} and is therefore omitted.
\end{remark}

\begin{remark}
	In \cite{Kr19}, we work under a slightly more general set-up. There we consider a larger space $\cX$ in which all $X_n$ and $X$ are mapped using maps $\eta_n$ and $\eta$. The author expects the Gamma convergence results to extend to that setting, but various results on Gamma-convergence need to be extended to that context. These extensions are not the focus of this paper, and we therefore restrict to this slightly simpler setting. Note that this is also the set-up of \cite{FK06} that is used in the context of large deviations.
\end{remark}

Conditions (b) should be interpreted in the sense that $K^q$ is larger than the `limit' of the sequence $K_n$, whereas (c) should be interpreted in the sense that each compact $K$ in $X$ is contained in a limit of that type.

We will say that a sequence $x_n \in X_n$ converges to $x \in X$ in the sense that $\eta_n(x_n) \rightarrow x$ in $X$. Using this notion of convergence, we can extend our notion of buc convergence.

\begin{definition} \label{definition:abstract_LIM}
	Let Assumption \ref{assumption:abstract_spaces_q} be satisfied. For each $n$ let $f_n \in M_b(X_n)$ and $f \in M_b(X)$. We say that $\LIM f_n = f$ if
	\begin{itemize}
		\item $\sup_n  \vn{f_n} < \infty$,
		\item if for all $q \in \cQ$ and $x_n \in K_n^q$ converging to $x \in K^q$ we have
		\begin{equation*}
		\lim_{n \rightarrow \infty} \left|f_n(x_n) - f(x)\right| = 0.
		\end{equation*}
	\end{itemize}
\end{definition}

\begin{remark} \label{remark:LIM_equivalence}
	Note that if $f \in C_b(X)$ and $f_n \in M_b(X_n)$, we have that $\LIM f_n = f$ if and only if
	\begin{itemize}
		\item $\sup_n  \vn{f_n} < \infty$,
		\item if for all $q \in \cQ$ 
		\begin{equation*}
		\lim_{n \rightarrow \infty} \sup_{x \in K_n^q} \left|f_n(x) - f(\eta_n(x))\right| = 0.
		\end{equation*}
	\end{itemize}
\end{remark}

\begin{remark} \label{remark:approximating_functions}
Note that for $f \in C_b(X)$, we have $f_n := f \circ \eta_n \in C_b(X)$ and $\LIM f_n = f$.
\end{remark}

Next, we extend our notions of equi-coercivity. This will be the notion of equi-coercivity that will be referred to later on in the paper.

\begin{definition} \label{definition:equi_coercive}
Let $\{X_n\}_{n\geq 1}$, $X$ be a collection of spaces satisfying Assumption \ref{assumption:abstract_spaces_q}. We say that a sequence of lower semi-continuous functionals $I_n : X_n \rightarrow [0,\infty]$ is \textit{equi-coercive} if for every $M \geq 0$ there exists $q \in \cQ$ such that for all $n \geq 1$, we have
\begin{equation*}
\left\{x \in X_n \, \middle| \, I^n(x) \leq M\right\} \subseteq  K^q_n.
\end{equation*}
\end{definition} 

\begin{remark}
The notion of equi-coercivity of Definition \ref{definition:equi_coercive} is stronger than the one of Definition \ref{definition:equi_coercive_basic}. Indeed, suppose that $\{I^n\}_{n \geq 1}$ are equi-coercive in the sense of Definition \ref{definition:equi_coercive}, then for every $M \geq 0$ there exists $q \in \cQ$ such that the set
\begin{equation*}
\bigcup_{n \geq 1} \left\{\eta_n(x) \, \middle| \, x \in X_n, \, I_n(x) \leq M\right\} \subseteq K^q,
\end{equation*}
and is, as a consequence, relatively compact in $X$. We conclude that the contractions
\begin{equation*}
J^n(y) = \inf\left\{I^n(x) \, | \, \eta_n(x) = y \right\}
\end{equation*}
of $I_n$ to $\eta_n(X_n)$ are equi-coercive on $X$ in the sense of Definition \ref{definition:equi_coercive_basic}. There are various contexts in which the two notions are equivalent. In particular this happens if the maps $\eta_n$ are proper and the sets $K_n^q$ are defined as $\eta_{n}^{-1}(K^q)$ and the set $\cQ$ is indexed by the compact sets in $X$.
\end{remark}

\begin{definition} \label{definition:Gamma_convergence}
	Let $\{X_n\}_{n\geq 1}$, $X$ be a collection of spaces satisfying Condition \ref{assumption:abstract_spaces_q}. Let $I^n : X_n \rightarrow [0,\infty]$ and $I : X \rightarrow [0,\infty]$ be functionals.
	\begin{enumerate}[(a)]
		\item We say that the Gamma convergence lower bound holds, i.e. $I \leq \Gamma-\liminf_n I^n$ if for all $x_n$ such that $\eta_n(x_n)  \rightarrow x$ it holds that
		\begin{equation*}
		I(x) \leq  \liminf_n I^n(x_n).
		\end{equation*}
		\item We say that the Gamma convergence upper bound holds, i.e. $I \geq \Gamma-\limsup_n I^n$ if for each $x$ there are $x_n$ such that $\eta_n(x_n) \rightarrow x$ and
		\begin{equation*}
		I(x) \geq \limsup_n I^n(x_n).
		\end{equation*}
		\item We say that $I^n$ Gamma converges to $I$, i.e. $I = \Gamma-\lim I^n$ if the upper and lower bound hold.
	\end{enumerate}
\end{definition}

\subsection{Path-spaces}

In the theorem \ref{theorem:main_theorem_basic} we established a Gamma convergence result for the space of continuous trajectories in $X$. As in the theory of large deviations and weak convergence of Markov processes, however, it is possible to extend these results to a more general context. Our result below is in some sense lacking, due to a failure to properly construct a recovery sequence in the context of discontinuous trajectories.

It should, however, be possible to prove a full result. To facilitate further study, we single out the key properties of the space of continuous trajectories that are needed to establish Gamma convergence via the finite-dimensional functionals.

We first establish the structure of the type of spaces of trajectories that we will work with, afterwards we will fix topological properties that allow us to work via a projective limit structure.

\begin{definition}[Path-space]
We say that $\Omega$ is a \textit{path-space on $X$} if $\Omega \subseteq \prod_{t \geq 0} X$ (here $\Omega$ is only considered as a set, not as a topological space), and if
\begin{enumerate}[(a)]
\item for $\gamma \in \Omega$ and $T \geq 0$, the trajectory $\gamma^T$ defined by $\gamma^T(t) = \gamma(T + t)$ is in $\Omega$,
\item for any two paths $\gamma,\hat{\gamma} \in \Omega$ and $T_1,T_2 \geq 0$, if $\gamma(T_1) = \hat{\gamma}(T_2)$, then the trajectory $\gamma^*$ defined by
\begin{equation*}
\gamma^*(t) = \begin{cases}
\gamma(t) & \text{for } t \leq T_1, \\
\hat{\gamma}(T_2 + (t-T_1)) & \text{for } t \geq T_1,
\end{cases}
\end{equation*}
is in $\Omega$.
\end{enumerate}
For $\gamma \in \Omega$, we denote by $\Delta_\gamma \subseteq \bR^+$ the set of points where $\gamma : \bR^+ \rightarrow X$ is discontinuous. We denote by $\pi_t : \Omega \rightarrow X$ the map $\pi_t(x) = x(t)$.
\end{definition}

Two main examples are given by the space of continuous trajectories $C_X(\bR^+)$ and the Skorokhod space $D_X(\bR^+)$.

We now turn to the topology on the path-space. In both settings, the topology, if restricted to compact sets is determined by the finite dimensional marginals. For the Skorokhod space, the situation is slightly more involved as the trajectories have discontinuities. In this context, one carefully needs to avoid the points of discontinuity. 
In both contexts however, a combination of restricting to compact sets and proving convergence for the finite dimensional marginals is feasible. For example, in large deviation theory this allows one to prove results via a projective limit theorem and a inverse contraction principle. In weak convergence theory for Markov processes similar results are known. Also in the context of Gamma convergence a inverse contraction principle and projective limit theorem are provable. We will not do so in this paper explicitly, but work with their effective results.

We thus assume the following topological properties for our path-space. 

\begin{assumption} \label{assumption:path_space}
The topology on the path-space $\Omega$ satisfies:
\begin{enumerate}[(a)]
\item \label{item:ass:metrizable_compacts} The compact subsets of $\Omega$ and $X$ are metrizable.
\item \label{item:projection_compacts} For every compact set $K \subseteq \Omega$ and $T \geq 0$, there exists a compact set $\hat{K} \subseteq X$ such that $\pi_t(K) \subseteq \hat{K}$ for all $t \leq T$.
\item \label{item:ass:open_sets_in_compacts} Let $\cT_0$ be dense in $[0,\infty)$ and let $\gamma \in \Omega$. For every open set $U \subset \Omega$ containing $\gamma$ and compact set $K \subseteq \Omega$, there exist $t_1,\dots,t_k \in \cT_0$ and an open set $U' \subseteq X^k$ containing $(\gamma(t_1),\dots,\gamma(t_k))$ such that
\begin{equation*}
\left\{y \in \Omega \, | \, (y(t_1),\dots,y(t_k)) \in U' \right\} \cap K \subseteq U.
\end{equation*}
\item \label{item:ass:existence_of_paths} For each finite collection of times $t_1,\dots,t_k$ and points $x_1,\dots,x_k \in X$ there is a path $\gamma \in \Omega$ such that $\gamma(t_i) = y_i$ for all $i$.
\end{enumerate}
\end{assumption}

\begin{lemma} \label{lemma:Skorokhod_space_properties}
Let $X$ be a complete separable metric space. The space $C_X(\bR^+)$ with the compact-open topology and the Skorokhod space $D_X(\bR^+)$ with the usual Skorokhod topology, cf. \cite{EK86}, satisfy Assumption \ref{assumption:path_space}.
\end{lemma}

\begin{proof}
	We give references for the Skorokhod topology as the Skorokhod topology restricted to $C_X(\bR^+)$ simplifies to the compact open topology. (a) is immediate as $X$ and $D_X(\bR^+)$ are equipped with a metric. (b) is proven in Theorem 3.6.3 in \cite{EK86}. (c) is proven in Lemma 4.26 in \cite{FK06}.
\end{proof}

\begin{definition}[Finite-dimensional determination of path-space functional 2]  \label{definition:representation_rate_function1}
	Let $\Omega$ be a path-space. We say that $I : \Omega\rightarrow [0,\infty]$ is a \textit{path-space functional} if it has projective limit form:
	\begin{equation} \label{eqn:functional_project2}
	I(\gamma)  = \sup_{k \geq 0} \quad \sup_{\substack{0 = t_0 < t_1 < \dots < t_k \\ t_i \notin \Delta_\gamma}} I[t_0,\dots,t_k](\gamma(t_0),\dots,\gamma(t_k)), 
	\end{equation}
	where
	\begin{equation}\label{eqn:functional_contract2}
	I[t_0,\dots,t_k](x_0,\dots,x_k) = \inf \left\{I(\gamma) \, \middle| \, \gamma \in \Omega, \forall i \in \{0,\dots,k\}: t_i \notin \Delta_\gamma, \gamma(t_i) = x_i\right\}.
	\end{equation}
\end{definition}

The definition of a path-space functional $I$ determined by a semigroup $V(t)$ and initial functional $I_0$ remains unchanged from Definition \ref{definition:representation_rate_function2_basic}, that is, if 
 if
	\begin{equation*} 
	I[t_0,\dots,t_k](x_0,\dots,x_k) = I_0(x_0) + \sum_{i=1}^k I_{t_i - t_{i-1}}(x_i \, | \, x_{i-1}).
	\end{equation*}
	where
	\begin{equation*}
	I_t(y \, | \, x) := \sup_{f \in C_b(X)} \left\{f(y) - V(t)f(x) \right\}.
	\end{equation*}

\section{Main results} \label{section:main_results}

\subsection{From strong convergence of dual functionals to Gamma convergence on the path space} \label{section:semigroup_convergence_to_Gamma_convergence}

Our main result is a Gamma convergence for functionals on a sequence of path-spaces.

Let $\{X_n\}_{n \geq 1}$ and $X$ be a collection of spaces satisfying Assumption \ref{assumption:abstract_spaces_q} and let $\Omega_n, \Omega$ be path-spaces on these spaces. The maps $\eta_n,\eta$ naturally induce continuous maps $\eta_n : \prod_{t \geq 0} X_n \rightarrow \prod_{t \geq 0} X$.

\begin{assumption}[Main setting] \label{assumption:main_condition}
Let $\{X_n\}_{n \geq 1}$ and $X$ be spaces satisfying Assumption \ref{assumption:abstract_spaces_q}. Let $\{\Omega_n\}_{n \geq 1}$ and $\Omega$ be the corresponding path-spaces satisfying Assumption \ref{assumption:path_space}. 
Denote by $\eta_n : \Omega_n \rightarrow \Omega$ the induced maps arising from $\eta_n : X_n \rightarrow X$.

\end{assumption}

We give two additional definitions that extend the notion of equi-coercivity concerning the compactness of the level sets of path-space functionals.

\begin{definition} \label{definition:compact_containment}
Let Assumption \ref{assumption:main_condition} be satisfied. Let $I^n : \Omega_n \rightarrow [0,\infty]$ and $I : \Omega \rightarrow [0,\infty]$ be path-space functionals.

\begin{enumerate}[(a)]
\item We say that the sequence $\{I^n\}_{n\geq 1}$ satisfies the \textit{compact containment condition} if for all $T \geq 0$ and $M \geq 0$, there exists a $q = q(T,M) \in \cQ$ such that for any $n$ it holds that if $I^n(\gamma) \leq M$, then $\gamma(t) \in  K_n^q$ for all $t \leq T$.
\item We say that the sequence $\{I^n\}_{n\geq 1}$ is \textit{equi-coercive} if for every $M \geq 0$, the set
\begin{equation*}
\bigcup_n \left\{\eta_n(\gamma) \, \middle| \, \gamma \in \Omega_n, I^n(\gamma) \leq M  \right\}
\end{equation*}
is relatively compact in $\Omega$.
\end{enumerate}
\end{definition}

Note that (a) and (b) are both small adaptations of Definition \ref{definition:equi_coercive_basic}. The first is for the product topologies on $\Omega_n$ and $\Omega$ arising from $X_n$ and $X$, whereas (b) is for the path-space topology. 

A minor additional difference is the change of what compact sets to work with. For (a) we work with \textit{specific} compact sets on $X_n$ arising from $\cQ$ instead of working with compact sets in the image space $\Omega$ (constructed from X). In the cases that we can connect these two concepts, (b) implies (a). This does not hold in general, so in later results we assume both properties.

\begin{lemma} \label{lemma:equi_coercive_implies_compact_containment}
	Let Assumption \ref{assumption:main_condition} be satisfied. In addition, suppose that for all $q \in \cQ$ we have that $K_n^q = \eta_n^{-1}(K^q)$. Let $I^n : \Omega_n \rightarrow [0,\infty]$ and $I : \Omega \rightarrow [0,\infty]$ be path-space functionals. Suppose that $\{I^n\}$ are equi-coercive, then they satisfy the compact containment condition.
\end{lemma}

We prove this Lemma in Section \ref{subsection:equi_coercivity_compact_containment}. We proceed with our main theorems on Gamma-convergence.

\begin{theorem} \label{theorem:Gamma_convergence_via_semigroup_convergence}
Let Assumption \ref{assumption:main_condition} be satisfied. Let $I^n : \Omega_n \rightarrow [0,\infty]$ and $I : \Omega \rightarrow [0,\infty]$ be path-space functionals determined by semigroups $V_n(t)$ and $V(t)$ and initial functionals $I_0^n$ and $I_0$. 

Suppose that
\begin{enumerate}[(a)]
	\item The functionals $I^n$ are equi-coercive and satisfy the compact containment condition;
	\item $\Gamma-\lim I_0^n = I_0$;
	\item For all $t_n \rightarrow t$, $f \in C_b(X)$ and $f_n \in C_b(X_n)$ such that $\LIM f_n = f$, we have that
	\begin{equation} \label{eqn:semigroup_convergence}
	\LIM V_n(t_n) f_n = V(t)f;
	\end{equation}
\end{enumerate}
then we have that $\Gamma-\lim I^n = I$.
\end{theorem}

The result is based on Propositions \ref{proposition:lower_bound} and \ref{proposition:upper_bound} that state the upper and lower bound separately. A Gamma-convergence results for the interval $[0,T]$ or $[0,T)$ instead of $\bR^+$ can be obtained by using e.g. the contraction principle proven in \cite[Lemma 4.6]{Ma18} for the space $\hat{\Omega}$ and a time restricted version of $\hat{\Omega}$. Alternatively, one can use the methods developed below for a direct (but essentially the same) proof.

\subsection{From the convergence of generators to Gamma convergence  on path-spaces:preliminaries } \label{section:preliminaries_main_result_generators}

We now use Theorem \ref{theorem:Gamma_convergence_via_semigroup_convergence} to extend Theorem \ref{theorem:main_theorem_basic} to a more general setting on the basis on the convergence of solutions of Hamilton-Jacobi equations in \cite{Kr19}. 

\smallskip

In the context of problems that involve homogenisation or slow-fast systems, it often pays of to work with multi-valued Hamiltonians whose range naturally takes values in a space of functions with a domain that is larger. This larger domain takes into account a variable that we homogenise over or the `fast' variable.

\begin{assumption} \label{assumption:abstract_spaces2}
	We have spaces $X_n$ and $X, Y$ and continuous maps $\eta_n : X_n \rightarrow X_n$, $\widehat{\eta}_n : X_n \rightarrow Y$ and a continuous surjective map $\gamma : Y \rightarrow X$ such that the following diagram commutes: \\
	\begin{center}
	\begin{tikzpicture}
	\matrix (m) [matrix of math nodes,row sep=1em,column sep=4em,minimum width=2em]
	{
		{ }   & Y \\
		X_n & { } \\
		{ }   & X \\};
	\path[-stealth]
	(m-2-1) edge node [above] {$\widehat{\eta}_n$} (m-1-2)
	(m-2-1) edge node [below] {$\eta_n$} (m-3-2)
	(m-1-2) edge node [right] {$\gamma$} (m-3-2);
	\end{tikzpicture}
	\end{center}
	
	There is a directed set $\cQ$ (partially ordered set such that every two elements have an upper bound). For each $q \in \cQ$, we have compact sets $K_n^q \subseteq X_n$ a compact sets $K^q \subseteq X$ and $\widehat{K}^q \subseteq Y$ such that
	\begin{enumerate}[(a)]
		\item If $q_1 \leq q_2$, we have $K^{q_1} \subseteq K^{q_2}$, $\widehat{K}^{q_1} \subseteq \widehat{K}^{q_2}$ and for all $n$ we have $K_{n}^{q_1} \subseteq K_n^{q_2}$.
		\item \label{item:assumption_abstract_2_limit_compact} For all $q \in \cQ$ we have $\bigcup_n \widehat{\eta}_n(K_n^q) \subseteq \widehat{K}^q$.
		\item \label{item:assumption_abstract_2_exists_q} For each compact set $K \subseteq X$, there is a $q \in \cQ$ such that
		\begin{equation*}
		K \subseteq \liminf_n \eta_n(K_n^q).
		\end{equation*}
		\item \label{item:assumption_abstract_2_gamma_mapsintoeachother} We have $\gamma(\widehat{K}^q) \subseteq K^q$.
	\end{enumerate}
\end{assumption}

Note the subtle difference with Assumption \ref{assumption:abstract_spaces_q} in the sense that here \eqref{item:assumption_abstract_2_limit_compact} is written down in terms of convergence in $Y$, whereas \eqref{item:assumption_abstract_2_exists_q} is still written down in terms of convergence in $X$.

\begin{remark}
	Note that Assumption \ref{assumption:abstract_spaces_q} implies Assumption \ref{assumption:abstract_spaces_q}. Thus, in the context of Assumption \ref{assumption:abstract_spaces2}, we can use all notions, like e.g. $\LIM$ and equi-coercivity and compact containment, of the previous sections.
\end{remark}

\begin{example}[Reduction of the dimension] \label{example:LIM_dimension_reduction}
	Consider two spaces $X$ and $Z$ and let $Y := X \times Z$, $X_n := X \times Z$ with maps $\eta_n(x,z) = x$, $\hat{\eta}_n(x,z) = (x,z)$ and $\gamma(x,z) = x$.
	
	Assumption \ref{assumption:abstract_spaces2} is satisfied for example with $\cQ$ the collection of pairs of compact sets in $X$ and $Z$:
	\begin{equation*}
	\left\{(K_1, K_2) \, \middle| \, \forall \, K_1 \subseteq X, K_2 \subseteq Z \text{ compact}\right\},
	\end{equation*}
	and $K^{(K_1,K_2)}_n = K_1 \times K_2$, $K^{(K_1,K_2)} = K_1$ and $\widehat{K}^{(K_1,K_2)} = K_1 \times K_2$.
	
	We have $\LIM f_n = f$ if and only if $\sup_n \vn{f_n} < \infty$ and for all compact $K_1 \subseteq X$ and $K_2 \subseteq Z$ and sequences $(x_n,z_n) \in K_1 \times K_2$ and $x \in K_X$ such that $x_n \rightarrow x$, we have $f_n(x_n,z_n) \rightarrow f(x)$.
	
	Note that the dependence of $f_n$ on $z_n$ should vanish in the limit.
\end{example}

A first notion of a limit of Hamiltonians is given by the notion of an extended limit. This notion is essentially the extension of the convergence condition for generators from the setting of the Trotter-Kato approximation theorem to a more general context. The generalization is made to include operators defined on different spaces, and is also applicable to non-linear operators as well. See e.g. the works of Kurtz and co-authors \cite{EK86,Ku70,Ku73,FK06}.

We define this notion for the setting in which $X = Y$.

\begin{definition} \label{definition:extended_limit}
	Consider the setting of Assumption \ref{assumption:abstract_spaces_q}. Suppose that for each $n$ we have an operator $H_{n} \subseteq M_b(X_n) \times M_b(X_n)$. The \textit{extended limit} $ex-\LIM_n H_{n}$ is defined by the collection $(f,g) \in M_b(X)\times M_b(X)$ such that there exist $(f_n,g_n) \in H_{n}$ with the property that $\LIM_n f_n = f$ and $\LIM_n g_n = g$.
\end{definition}

We aim to have a more flexible notion of convergence by replacing all operators $H_n$ and $H$ by operators $(H_{n,\dagger},H_{n,\ddagger},H_\dagger,H_\ddagger)$ that intuitively form natural upper and lower bounds for $H_n$ and $H$. We will also generalize by considering limiting Hamiltonians that take values in the set of functions on $Y$ instead of $X$.

\begin{definition} \label{definition:extended_sub_super_limit}
	Consider the setting of Assumption \ref{assumption:abstract_spaces2}.  Suppose that for each $n$ we have two operators $H_{n,\dagger} \subseteq LSC_l(X_n) \times USC_u(X_n)$ and $H_{n,\ddagger} \subseteq USC_u(X_n) \times LSC_l(X_n)$.
	\begin{enumerate}[(a)]
		\item The \textit{extended sub-limit} $ex-\subLIM_n H_{n}$ is defined by the collection $(f,g) \in H_\dagger \subseteq LSC_l(X) \times USC_u(Y)$ such that there exist $(f_n,g_n) \in H_{n,\dagger}$ satisfying
		\begin{gather} 
		\LIM f_n \wedge c = f \wedge c, \qquad \forall \, c \in \bR, \label{eqn:convergence_condition_sublim_constants} \\
		\sup_{n} \sup_{x \in X_n} g_n(x) < \infty, \label{eqn:convergence_condition_sublim_uniform_gn}
		\end{gather}
		and if for any $q \in \cQ$ and sequence $z_{n(k)} \in K_{n(k)}^q$ (with $k \mapsto n(k)$ strictly increasing) such that $\lim_{k} \widehat{\eta}_{n(k)}(z_{n(k)}) = y$ in $Y$ with $\lim_k f_{n(k)}(z_{n(k)}) = f(\gamma(y)) < \infty$ we have
		\begin{equation} \label{eqn:sublim_generators_upperbound}
		\limsup_{k \rightarrow \infty}g_{n(k)}(z_{n(k)}) \leq g(y).
		\end{equation}
		\item The \textit{extended super-limit} $ex-\superLIM_n H_{n}$ is defined by the collection $(f,g) \in H_\ddagger \subseteq USC_u(X) \times LSC_l(Y)$ such that there exist $(f_n,g_n) \in H_{n,\ddagger}$ satisfying
		\begin{gather} 
		\LIM f_n \vee c = f \vee c, \qquad \forall \, c \in \bR, \label{eqn:convergence_condition_superlim_constants} \\
		\inf_{n} \inf_{x \in X_n} g_n(x) > - \infty, \label{eqn:convergence_condition_superlim_uniform_gn}
		\end{gather}
		and if for any $q \in \cQ$ and sequence $z_{n(k)} \in K_{n(k)}^q$ (with $k \mapsto n(k)$ strictly increasing) such that $\lim_{k} \widehat{\eta}_{n(k)}(z_{n(k)}) =y$ in $Y$ with $\lim_k f_{n(k)}(z_{n(k)}) = f(\gamma(y)) > - \infty$ we have
		\begin{equation}\label{eqn:superlim_generators_lowerbound}
		\liminf_{k \rightarrow \infty}g_{n(k)}(z_{n(k)}) \geq g(y).
		\end{equation}
	\end{enumerate}
	
\end{definition}

\begin{remark}
	The conditions in \eqref{eqn:convergence_condition_sublim_uniform_gn} and \eqref{eqn:sublim_generators_upperbound} are implied by $\LIM_n g_n \leq g$ and \eqref{eqn:convergence_condition_superlim_uniform_gn} and \eqref{eqn:superlim_generators_lowerbound} are implied by $\LIM_n g_n \geq g$.
\end{remark}

\subsection{From convergence of Hamiltonians to Gamma convergence}  \label{section:main_result_generators}

We have now developed the machinery to state the main result of this paper.

\begin{condition} \label{condition:convergence_of_generators_and_conditions_extended_supsuperlim}
	Consider the setting of Assumption \ref{assumption:abstract_spaces2}. 
	
	There are sets $B_n$ such that $C_b(X_n) \subseteq B_n \subseteq M(X_n)$, contractive pseudo-resolvents $R_n(\lambda) : B_n \rightarrow B_n$, $\lambda >0$, with $R_n(\lambda) C_b(X_n) \subseteq C_b(X_n)$ and contractive semigroups $V_n(t) : C_b(X_n) \rightarrow C_b(X_n)$ generated by $R_n(\lambda)$. In addition, suppose we have operators
	\begin{align*}
	H_{n,\dagger} & \subseteq \left(LSC_l(X_n) \cap B_n\right)  \times \left( USC_u(X_n) \cap B_n\right), \\
	H_{n,\ddagger} & \subseteq \left( USC_u(X_n) \cap B_n\right) \times \left( LSC_l(X_n) \cap B_n\right). 
	\end{align*}
	These spaces and operators have the following properties:
	\begin{enumerate}[(a)]
		\item \label{item:convH_pseudoresolvents_solve_HJ} For all $n \geq 1$, $\lambda > 0$ and $h \in B_n$ the function $(R_n(\lambda)h)^*$ is a viscosity subsolution to $f -  \lambda H_{n,\dagger} = h$. Similarly, $(R_n(\lambda)h)_*$ is a viscosity supersolution to $f - \lambda H_{n,\ddagger} f = h$.
		\item \label{item:convH_strict_equicont_resolvents} We have local strict equi-continuity on bounded sets for the resolvents: for all $q \in \cQ$, $\delta > 0$ and $\lambda_0 > 0$, there is a $\hat{q} \in \cQ$ such that for all $n$ and $h_{1},h_{2} \in B_n$ and $0 < \lambda \leq \lambda_0$ that
		\begin{multline*}
		\sup_{y \in K_n^q} \left\{ R_n(\lambda)h_{1}(y) - R_n(\lambda)h_{2}(y) \right\} \\
		\leq \delta \sup_{x \in X_n} \left\{ h_{1}(x) - h_{2}(x) \right\} + \sup_{y \in K^{\hat{q}}_n} \left\{ h_{1}(y) - h_{2}(y) \right\}.
		\end{multline*}
		\item \label{item:convH_strict_equicont_semigroups}  We have local strict equi-continuity on bounded sets for the semigroups: for all $q \in \cQ$, $\delta > 0$ and $t_0 > 0$, there is a $\hat{q} \in \cQ$ such that for all $n$ and $h_{1},h_{2} \in B_n$ and $0 \leq t \leq t_0$ that
		\begin{multline*}
		\sup_{y \in K_n^q} \left\{ V_n(t)h_{1}(y) - V_n(t)h_{2}(y) \right\} \\
		 \leq \delta \sup_{x \in X_n} \left\{ h_{1}(x) - h_{2}(x) \right\} + \sup_{y \in K^{\hat{q}}_n} \left\{ h_{1}(y) - h_{2}(y) \right\}.
		\end{multline*}
	\end{enumerate}
\end{condition}
%

The main corollary of this result is a generator version of Theorem \ref{theorem:Gamma_convergence_via_semigroup_convergence} which generalizes Theorem \ref{theorem:main_theorem_basic}. 


\begin{theorem} \label{theorem:Gamma_convergence}
	Let Assumption \ref{assumption:main_condition} be satisfied. Let $I^n : \Omega_n \rightarrow [0,\infty]$ and $I : \Omega \rightarrow [0,\infty]$ be path-space functionals determined by semigroups $V_n(t)$ and $V(t)$ and initial functionals $I_0^n$ and $I_0$. 
	
	Suppose that
	\begin{enumerate}[(a)]
		\item The functionals $I^n$ are equi-coercive and satisfy the compact containment condition;
		\item $\Gamma-\lim I_0^n = I_0$;
		\item Condition \ref{condition:convergence_of_generators_and_conditions_extended_supsuperlim} is satisfied.
		\item \label{item:convH_sub_superLIM} $H_\dagger \subseteq ex-\subLIM H_{n,\dagger}$ and $H_\ddagger \subseteq ex-\superLIM H_{n,\ddagger}$;
		\item Let $D \subseteq C_b(X)$ be quasi-dense in $C_b(X)$. Suppose that for each $\lambda > 0$ and $h \in D$ the comparison principle holds for 
		\begin{equation*} 
		f - \lambda H_\dagger f = h, \qquad f - \lambda H_\ddagger f = h.
		\end{equation*}
		Denote the unique solution by $R(\lambda)h$.
		\item For all $h \in D$ and $\lambda > 0$ It holds that $R(\lambda)h$ buc converges to $h$ (bounded and uniform on compacts).
	\end{enumerate}
Then we have $I = \Gamma-\lim I^n$.
\end{theorem}

\begin{remark}
	For the Gamma lower bound, one does not need equi-coercivity of $I^n$. It suffices that the compact containment condition holds or that for every $x \in X$ there is a $q \in \cQ$ and $x_n \in K_n^q$ such that $\eta_n(x_n) \rightarrow x$.
\end{remark}

\section{Proofs of results in Section \ref{section:main_results}} \label{section:proofs_main_results}

In this section, we will prove the main results of this paper. We start with some auxiliary results on equi-coercivity and the compact containment condition in Section \ref{subsection:equi_coercivity_compact_containment}. We proceed with proofs of the lower and upper bound in Sections \ref{subsection:proof_lower_bound} and \ref{subsection:proof_upper_bound}. In the final section \ref{subsection:proof_main_result} we establish the main results, Theorem \ref{theorem:Gamma_convergence_via_semigroup_convergence} and \ref{theorem:Gamma_convergence}.

\subsection{Some remarks on equi-coercivity and compact containment} \label{subsection:equi_coercivity_compact_containment}
%
%
%
%
%
%
%
%

We start with the proof of Lemma \ref{lemma:equi_coercive_implies_compact_containment} 

\begin{proof}[Proof of Lemma \ref{lemma:equi_coercive_implies_compact_containment}]
	Fix $T \geq 0$ and $M \geq 0$. By Assumption \ref{assumption:path_space} \eqref{item:projection_compacts} and the equi-coercivity of the functionals $I^n$, there is a compact set $K \subseteq X$ such that if $t \leq T$, $\gamma \in \Omega_n$ such that $I^n(\gamma) \leq M$, then $\eta_n(\gamma(t)) \in K$. By Assumption \ref{assumption:abstract_spaces_q} (b) and (c), we find that there is some $q$ such that $K \subseteq K^q$. Thus, if $K^q_n = \eta_n^{-1}(K^q)$, we indeed have that $\gamma(t) \in K_n^q$
\end{proof}

The proof of the following lemma is straightforward.

\begin{lemma} \label{lemma:contraction_coercivity}
Let $\cX,\cY$ be two spaces and let $\cI : \cX \rightarrow [0,\infty]$ be coercive and lower semi-continuous. Let $\pi : \cX \rightarrow \cY$ be a map that is continuous on the sublevel sets of $\cI$.
Define
\begin{equation*}
\cJ(y) = \inf \left\{\cI(x) \, \middle| \, x \in \cX, \, \pi(x) = y \right\}.
\end{equation*}
Then $\cJ$ is lower semi-continuous and coercive.
\end{lemma}

A variant of this lemma holds for equi-coercivity of a sequence of functionals. We are interested in projection maps on the path-space. As we will be projecting on finite dimensional distributions, we it suffices to assume the compact containment condition instead of equi-coercivity. Finally, we will work in the context of sequences of compact sets based on the index set $\cQ$.

\begin{lemma} \label{lemma:contracting_equi_coercivity}
Let Assumption \ref{assumption:abstract_spaces_q} be satisfied and let $0 =t_0 < t_1 < \dots t_k$, $k \geq 0$ be a collection of times. Let $\pi_{t_0,\dots,t_k}^n : \Omega_n \rightarrow X_n^{k+1}$ and $\pi_{t_0,\dots,t_k} : \Omega \rightarrow X^{k+1}$ be the projection maps
\begin{equation*}
\gamma \mapsto (\gamma(t_0),\dots,t(\gamma_k)).
\end{equation*}
If the sequence $\{I^n\}_{n \geq 1}$ satisfies the compact containment condition, then the functionals
\begin{equation*}
I^n[t_0,\dots,t_k](y_0,\dots,y_k) = \inf \left\{I^n(\gamma) \, \middle| \, \gamma \in \Omega_n, \, \pi_n(\gamma) = (y_0,\dots,y_k) \right\}
\end{equation*}
are equi-coercive on $X^{k+1}$, i.e. there is some $q \in \cQ$ such that for all $n \geq 1$
\begin{equation*}
\left\{y \in X^{k+1} \, \middle| \, I^n[t_0,\dots,t_k](y_0,\dots,y_k) \leq M \right\} \subseteq (K_n^q)^{k+1}.
\end{equation*}
\end{lemma}

%
%
%

\subsection{Proof of the lower bound} \label{subsection:proof_lower_bound}

In this section we assume Assumption \ref{assumption:main_condition}. In addition throughout this Section $I^n : \Omega_n \rightarrow [0,\infty]$ and $I : \Omega \rightarrow [0,\infty]$ are path-space functionals determined by semigroups $V_n(t)$ and $V(t)$ and initial functionals $I_0^n$ and $I_0$. 

\begin{proposition}[The lower bound] \label{proposition:lower_bound}
Let $D \subseteq C_b(X)$ be set that is bounded from above and isolates points. Suppose that either (i) or (ii) holds:
\begin{enumerate}[(i)]
	\item for all $t \geq 0$, $f \in D$ and sequences $t_n \rightarrow t$ and $f_n \in C_b(X_n)$ such that $\LIM f_n = f$, we have that
	\begin{equation*}
	\LIM V_n(t_n) f_n = V(t)f.
	\end{equation*}
	\item for all $ 0 \leq s \leq t$, $f \in D$ and sequences $f_n \in C_b(X_n)$ such that $\LIM f_n = f$, we have that
	\begin{equation*}
	\LIM V_n(t) f_n = V(t)f.
	\end{equation*}
	Additionally, for each $n$ and path $\gamma \in \Omega_n$ such that $I^n(\gamma) < \infty$, we have that $\Delta_{\gamma}^c = \bR^+$.
\end{enumerate}
Suppose that (a) or (b) holds:
\begin{enumerate}[(a)]
\item the compact containment holds for $\{I^n\}_{n \geq 1}$,
\item for all $x \in X$ there is a $q \in \cQ$ and $x_n \in K_n^q$ such that $\eta_n(x_n) \rightarrow x$.
\end{enumerate}
If $\Gamma-\underline{\lim} \, I_0^n \geq I_0$, then we have for all $\gamma \in \Omega$ and sequences $\gamma_n \in \Omega_n$ such that $\eta_n(\gamma_n) \rightarrow \gamma$:
\begin{equation*}
\liminf_{n \rightarrow \infty} \, I_n(\gamma_n) \geq I(\gamma).
\end{equation*}
\end{proposition}

The proposition will be derived from a similar result on the finite dimensional functionals. Recall that $O_k := \left\{(s_0,\dots,s_k) \in (\bR^+)^{k+1} \, \middle| \, 0 = s_0 < s_1 \dots < s_k \right\}$.

\begin{lemma}[The upper bound for the finite-dimensional functionals] \label{lemma:lower_bound_finite_dim}
	Suppose that $D \subseteq C_b(X)$ is a set that is bounded from above and isolates points.
	
	Suppose that for all $f \in D$ $t \geq 0$ and all sequences $t_n \geq 0$, $f_n \in C_b(X_n)$ such that $\LIM f_n = f$ and $t_n \rightarrow t$, we have that
	\begin{equation*}
	\LIM V_n(t_n) f_n = V(t)f.
	\end{equation*}
	Finally, suppose that (a) or (b) holds:
	\begin{enumerate}[(a)]
		\item the compact containment holds for $\{I^n\}_{n \geq 1}$,
		\item for all $x \in X$ there is a $q \in \cQ$ and $x_n \in K_n^q$ such that $\eta_n(x_n) \rightarrow x$.
	\end{enumerate}
	
	Let $(t_0,\dots,t_k) \in O_k$ and let $(t_{0,n},\dots,t_{k,n}) \in O_k$ be such that $t_{i,n} \rightarrow t_i$. 
	
	If $\Gamma-\underline{\lim} \, I_0^n \leq I_0$, then we have $\Gamma-\underline{\lim} \, I^n[t_{0,n},\dots,t_{k,n}] \leq I[t_0,\dots,t_k]$. 
\end{lemma}

The lemma will proven using an induction step based on the following lemma. In this lemma we establish a Gamma lower bound for the conditional functionals $I_t$ that appear in the representation \eqref{eqn:functional_finite_d_Markov}.  

\begin{lemma} \label{lemma:conditional_lower_bound}
Fix $x,y \in X$. Suppose that $D \subseteq C_b(X)$ is a set that is bounded from above and isolates points. Let $t \geq 0$ and let $t_n \rightarrow t$. 

Suppose that for all $f \in D$ and all sequences $f_n \in C_b(X_n)$ such that $\LIM f_n = f$, we have that
\begin{equation} \label{lemma:cond_lower_bound_semigroup_convergence}
\LIM V_n(t_n) f_n = V(t)f.
\end{equation}

Consider sequences $x_n \in X_n$ and $y_n \in X_n$ and points $x,y \in X$ such that $\eta_n(x_n) \rightarrow x$ and $\eta_n(y_n) \rightarrow y$ with the property that there are $q,q' \in \cQ$ with $x_n \in K_n^q$ and $y_n \in K_n^{q'}$ for all $n$. Then we have:
\begin{equation*}
\liminf_{n \rightarrow \infty} I_{t_n}^n(y_n \, | \, x_n) \geq I_{t}(y \, | \, x).
\end{equation*}
\end{lemma}

\begin{proof}
Fix $x,y \in X$ and consider sequences $x_n \in X_n$ and $y_n \in X_n$ such that $\eta_n(x_n) \rightarrow x$ and $\eta_n(y_n) \rightarrow y$ with the property that there are $q,q' \in \cQ$ with $x_n \in K_n^q$ and $y_n \in K_n^{q'}$ for all $n$. By Proposition \ref{proposition:bounded_above_and_isolating_points_implies_functional_determining}, we have
\begin{equation*}
I_{t}(y \, | \, x) = \sup_{f \in D} f(y) - V(t)f(x).
\end{equation*}

First suppose that $I_{t}(y \, | \, x) < \infty$. For each $\varepsilon > 0$, let $f \in D$ be such that
\begin{equation*}
I_{t}(y \, | \, x) \leq f(y) - V(t)f(x) + \varepsilon.
\end{equation*}
By Remark \ref{remark:approximating_functions}, there is a sequence $f_n \in C_b(X_n)$ such that $\LIM f_n = f$. By \eqref{lemma:cond_lower_bound_semigroup_convergence} and the existence of $q,q'$, we find that
\begin{align*}
I_{t}(y \, | \, x) & \leq \lim_{n \rightarrow \infty} f_n (y_n) - V_n(t_n)f_n(x_n) + \varepsilon \\
& \leq \liminf_{n \rightarrow \infty} I_{t_n}^n(y_n \, | \, x_n) + \varepsilon. 
\end{align*}

As $\varepsilon > 0$ was arbitrary, this proves the case where $I_{t}(y \, | \, x) < \infty$.

\smallskip

Now suppose that $I_{t}(y \, | \, x) = \infty$. For each $M > 0$, let $f_M \in D$ be such that
\begin{equation*}
M \leq f_M(y) - V(t)f_M(x).
\end{equation*}
Again by \eqref{lemma:cond_lower_bound_semigroup_convergence}, we can find for every $\varepsilon$ and $n$ sufficiently large a function $f_{M,n} \in C_b(X_n)$ such that
\begin{equation*}
M - \varepsilon \leq f_{M,n}(y_n) - V_n(t_n)f_{M,n}(x_n) \leq I_{t_n}^n(y_n \, | \, x_n).
\end{equation*}
As $\varepsilon$ and $M$ are arbitrary, we conclude that $\liminf_{n \rightarrow \infty} I_{t_n}^n(y_n \, | \, x_n) = \infty$.
\end{proof}

\begin{proof}[Proof of Lemma \ref{lemma:lower_bound_finite_dim}]
	The result follows by a straightforward application of Lemma \ref{lemma:conditional_lower_bound}. The condition on the existence of a $q_i \in \cQ$ such that $t_{i,n} \in K_n^{q_i}$ follows either from Condition (b) or from the combination of Condition (a) and Lemma \ref{lemma:contracting_equi_coercivity}. 
\end{proof}

We now turn to the proof of Proposition \ref{proposition:lower_bound}. The main challenge to

The following lemma is used to find times for which the time marginals converge. Recall that $O_k := \left\{(s_0,\dots,s_k) \in (\bR^+)^{k+1} \, \middle| \, 0 = s_0 < s_1 \dots < s_k \right\}$.

\begin{lemma} \label{lemma:convergence_for_marginals}
Let $Y$ be a completely regular space and let $\Omega'$ be a path-space satisfying Assumptions \ref{assumption:path_space} (a) and (b): compact sets in $Y$ and $\Omega'$ are metrizable and for all $T \geq 0$ and compact sets $K \subseteq \Omega'$, there exists $\hat{K} \subseteq  Y$ such that $\pi_t(K) \subseteq \hat{K}$ for all $t \leq T$. 

\smallskip

Suppose $\gamma_n \rightarrow \gamma$ in $\Omega'$ and fix times $(t_0,\dots,t_k) \in O_k$ for some $k \geq 1$. Then, we can find times $(t_{0,n},\dots t_{k,n}) \in O_k \cap (\Delta_{\gamma_n}^c)^{k+1}$ such that $t_{i,n} \rightarrow t_i$ and a subsequence $n'$ such that $\gamma_{n'}(t_{i,n'}) \rightarrow \gamma(t_i)$ for all $i$.

For each $n$ such that $\Delta_{\gamma_n}^c = \bR^+$, we can take $(t_{0,n},\dots t_{k,n})  = (t_0,\dots,t_k)$.
\end{lemma}

\begin{proof}
By an elementary argument, we can find $(t_{0,n},\dots t_{k,n}) \in O_k \cap (\Delta_{\gamma_n}^c)^{k+1}$ such that $t_{i,n} \rightarrow t_i$ for all $i$. A subsequence $n'$ such that $\gamma_{n'}(t_{i,n'}) \rightarrow \gamma(t_i)$ can be constructed via a diagonal argument using the compactness of $\hat{K}$ and the metrizability of this set.
\end{proof}

\begin{proof}[Proof of Proposition \ref{proposition:lower_bound}]
Fix $\gamma \in \Omega$ and let $\gamma_n$ be any sequence such that $\eta_n(\gamma_n) \rightarrow \gamma$. 

First suppose that $I(\gamma) < \infty$. Without loss of generality, we can restrict ourselves to a subsequence $n'$ such that $I_{n'}(\gamma_{n'}) < \infty$ for all $n'$ and $\lim_{n'} I_{n'}(\gamma_{n'}) = \liminf_n I_n(\gamma_n)$. Fix $\varepsilon > 0$. By the representation of $I$ in \eqref{eqn:functional_project}, we find $0 = t_0 < t_1 < \dots < t_k$ in $\Delta_\gamma^c$ such that
\begin{equation} \label{eqn:lower_bound_first_eq}
I(\gamma) \leq I_0(\gamma(0)) + \sum_{i = 1}^k I_{t_{i-1},t_i}(\gamma(t_i) \, | \, \gamma(t_{i-1})) + \varepsilon.
\end{equation}
By Lemma \ref{lemma:convergence_for_marginals}, we find times $(t_{0,n},\dots,t_{k,n}) \in O_k \cap \Delta_{\gamma_n}$ and a further subsequence $n''$ such that $\gamma_{n''}(t_{i,n''}) \rightarrow \gamma(t_i)$. This puts us in a position to apply Lemma \ref{lemma:lower_bound_finite_dim} along the subsequence $n''$:
\begin{multline}\label{eqn:lower_bound_second_eq}
I_0(\gamma(0)) + \sum_{i = 1}^k I_{t_{i-1},t_i}(\gamma(t_i) \, | \, \gamma(t_{i-1})) \\
\leq \liminf_{n'' \rightarrow \infty} I_0^{n''}(\gamma_{n''}(0)) + \sum_{i = 1}^k I_{t_{i-1,n''},t_{i,n''}}^{n''}(\gamma_{n''}(t_{i,n''}) \, | \, \gamma_{n''}(t_{i-1,n''})).
\end{multline}
Combining \eqref{eqn:lower_bound_first_eq}, \eqref{eqn:lower_bound_second_eq}, and representation \eqref{eqn:functional_project} for $I^{n''}$, gives
\begin{align*}
I(\gamma) & \leq \liminf_{n'' \rightarrow \infty} I_0^{n''}(\gamma_{n''}(0)) + \sum_{i = 1}^k I_{t_{i-1,n''},t_{i,n''}}^{n''}(\gamma_{n''}(t_{i,n''}) \, | \, \gamma_{n''}(t_{i-1,n''})) + \varepsilon \\
& \leq \liminf_{n'' \rightarrow \infty} I^{n''}(\gamma_{n''}) + \varepsilon.
\end{align*}
Because $n''$ was chosen to be a subsequence of a subsequence along which $\liminf_n I^n(\gamma_{n}) = \lim_{n'} I^{n'}(\gamma_{n'})$, we conclude that $I(\gamma) \leq \liminf_{n} I^n(\gamma_n) + \varepsilon$ for every $\varepsilon > 0$. As $\varepsilon > 0$ is arbitrary, the case that $I(\gamma) < \infty$ is proven.

\smallskip

The proof for $\gamma$ such that $I(\gamma) = \infty$ is similar but easier and is therefore omitted.
\end{proof}

\subsection{Proof of the upper bound} \label{subsection:proof_upper_bound}

As in last section, we assume Assumption \ref{assumption:main_condition}. In addition throughout this Section $I^n : \Omega_n \rightarrow [0,\infty]$ and $I : \Omega \rightarrow [0,\infty]$ are path-space functionals determined by semigroups $V_n(t)$ and $V(t)$ and initial functionals $I_0^n$ and $I_0$. 

\begin{proposition}[The upper bound]\label{proposition:upper_bound}
Suppose that the functionals $\{I_n\}_{n \geq 1}$ are equi-coercive and satisfy the compact containment condition. Let $D \subseteq C_b(X)$ be set that is bounded from above and isolates points such that for all $t \geq 0$, $f \in D$ and sequences $f_n \in C_b(X_n)$ $t_n \geq 0$ such that $\LIM f_n = f$ and $t_n \rightarrow t$, we have that
\begin{equation*}
\LIM V_n(t_n) f_n = V(t)f.
\end{equation*}
Then, we can find for all $\gamma \in \Omega$ a sequence $\gamma_n \in \Omega_n$ such that $\eta_n(\gamma_n) \rightarrow \gamma$ and
\begin{equation*}
\limsup_{n \rightarrow \infty} I_n(\gamma_n) \leq I(\gamma).
\end{equation*}
\end{proposition}

As for the lower bound, we start with the a upper bound for the finite-dimensional functionals. Recall that $O_k := \left\{(s_0,\dots,s_k) \in (\bR^+)^{k+1} \, \middle| \, 0 = s_0 < s_1 \dots < s_k \right\}$.

\begin{lemma}[The upper bound for the finite-dimensional functionals] \label{lemma:upper_bound_finite_dim}
	Suppose that Assumption \ref{assumption:main_condition} is satisfied and that $\{I^n\}$ satisfies the compact containment condition.	Let $D \subseteq C_b(X)$ be a set that is bounded from above and isolates points.

	Suppose that for all $f \in D$, $t \geq 0$ and all sequences $t_n \geq 0$, $f_n \in C_b(X_n)$ such that $\LIM f_n = f$ and $t_n \rightarrow t$, we have that
	\begin{equation*}
	\LIM V_n(t_n) f_n = V(t)f.
	\end{equation*}
	Let $(t_0,\dots,t_k) \in O_k$ and let $(t_{0,n},\dots,t_{k,n}) \in O_k$ be such that $t_{i,n} \rightarrow t_i$.	If $\Gamma-\overline{\lim} \, I_0^n \leq I_0$, then we have $\Gamma-\overline{\lim} \, I^n[t_{0,n},\dots,t_{k,n}] \leq I[t_0,\dots,t_k]$. 
\end{lemma}

As for the lower bound, we first prove an abstract result that we can use to prove a Gamma-convergence result via the conditional functionals that appear in the representation in \eqref{eqn:functional_finite_d_Markov}.

\begin{proposition} \label{proposition:conditional_upper_bound}
Let $\{\cX_n\}_{n \geq 1},\cX$ and $\{\cY_n\}_{n \geq 1},\cY$ be two collections of spaces satisfying Assumption \ref{assumption:abstract_spaces_q} with index sets $\cQ^\cX$ and $\cQ^\cY$.

Suppose we have lower semi-continuous coercive functionals
\begin{equation*}
\cI^n : \cX_n \times \cY_n \rightarrow [0,\infty],  \qquad \cI  : \cX \times \cY \rightarrow [0,\infty],
\end{equation*}
of the form
\begin{equation*}
\cI^n(x,y) = \cI_0^n(x) + \cJ^n(y \, | \, x), \qquad \cI(x,y) = \cI_0(x) + \cJ(y \, | \, x),
\end{equation*}
where
\begin{equation*}
\cI^n_0 : \cX_n \rightarrow [0,\infty], \qquad  \cI_0 : \cX \rightarrow [0,\infty], 
\end{equation*}
are lower semi-continuous and coercive and where we have maps $\Lambda_n : C_b(\cY_n) \rightarrow M_b(\cX_n)$ and $\Lambda : C_b(\cY) \rightarrow M_b(\cX)$ such that
\begin{equation*}
\cJ^n(y \, | \, x) = \sup_{f \in C_b(\cY_n)} f(y) - (\Lambda_n f)(x), \qquad \cJ(y \, | \, x)  = \sup_{f \in C_b(\cY)} f(y) - (\Lambda f)(x).
\end{equation*}

Suppose that $\cI_n$ are equi-coercive: for every $M \geq 0$, there are $q \in \cQ^\cX$ and $q' \in \cQ^\cY$ such that
\begin{equation} \label{eqn:lemma_conditional_Gamma_convergence_compact_containment}
\left\{(x,y) \in \cX_n \times \cY_n \, \middle| \, \cI^n(x,y) \leq M \right\} \subseteq K_n^q \times K_n^{q'}.
\end{equation}

Suppose that $\Gamma-\overline{\lim}_n \, \cI^n_0 \leq I_0$ and there is some set $D \subseteq C_b(\cY)$ that is bounded from above and isolates points and that for all $f \in D$ and $f_n \in C_b(\cX_n)$ such that $\LIM f_n = f$ we have 
\begin{equation} \label{eqn:convergence_of_loglaplace_in_conditionalupperbound}
\LIM \Lambda_n f_n = \Lambda f.
\end{equation}
Additionally suppose that for $f \in D$, we have $\Lambda f \in C_b(\cX)$.

Then $\Gamma-\overline{\lim} \, \cI^n \leq \cI$.
\end{proposition}

\begin{proof}
Fix $x \in \cX$, $y \in \cY$. Without loss of generality, we assume $\cI(x,y) < \infty$. 

We have to find two sequences $x_n \in \cX_n$ and $y_n \in \cY_n$ that yield the upper bound. The construction of $x_n$ is straightforward by using the $\Gamma$-upper bound $\Gamma-\overline{\lim}_n \, \cI_0^n \leq \cI_0$, which implies we can choose a sequence $x_n$ such that $\eta_n(x_n) \rightarrow x$ and
\begin{equation} \label{eqn:conditional_gamma_upperbound_I0}
\limsup_{n \rightarrow \infty} \, \cI^n_0(x_n) \leq \cI_0(x).
\end{equation}
Starting from this sequence, we work on the conditional functionals $\cJ^n$ to construct a sequence $y_n$ that works well in combination with the sequence $x_n$ and such that $\limsup_n \cJ^n(y_n \, | \, x_n) \leq \cJ(y \, | \,x)$. To do so, we will first express $\cJ$ in terms of $\Lambda f_m$ for a conveniently chosen sequence of functions $f_m \in D$, see Lemma \ref{lemma:isolating_points_gives_functions_that_approximate}. Thus, we first have to identify a suitable compact set.

\smallskip

Denote by $c := \sup_n I^n_0(x_n)$, which is finite by assumption, and let $C$ be the uniform upper bound for functions in $D$. Let $q \in \cQ^\cX$ and $q' \in \cQ^\cY$ be such that for all $n$
\begin{equation} \label{eqn:conditional_upper_bound_compact_set}
\left\{(x,y) \in \cX_n \times \cY_n \, \middle| \, \cI^n(x,y) \leq \cI(x,y) + c + C + 1 \right\} \subseteq K_n^q \times K_n^{q'}
\end{equation}
Denote $K := K^{q'}$. By assumption there is some metric $d$ on $K$. Denote by $B_r(z) \subseteq K$ the open(in the subspace topology on $K$) ball of radius $r$ around $z \in K$. By Lemma \ref{lemma:isolating_points_gives_functions_that_approximate} there are functions $\{f_m\}_{m \geq 1}$ in $D$ with designated point $y$ and compact set $K$ with the properties 
\begin{enumerate}[(a)]
\item $|f_m(y)| \leq m^{-1}$,
\item 
\begin{equation}\label{eqn:upper_bound_on_fm_on_K}
\sup_{z \in K} f_m(z) \leq 0.
\end{equation}
\item $\sup_{z \in K \cap B_{1/m}^c(y)} f_m(z) \leq - m$,
\item $\cJ(y \, | \, x) =  \lim_{m \rightarrow \infty} - (\Lambda f_m)(x)$.
\end{enumerate}

Fix $\varepsilon > 0$ (without loss of generality assume that $4 \varepsilon \leq 1$) and let $m = m(\varepsilon) \geq \varepsilon^{-1}$ be large enough such that 
\begin{enumerate}[(1)]
\item we have
\begin{equation} \label{eqn:function_fm_forces_in_small_set}
\sup_{z \in K \cap B_\varepsilon^c(y)} f_m(z) \leq - \cJ(y \, | \, x) - 1,
\end{equation}
\item we have
\begin{equation} \label{eqn:difference_V_and_I}
\left|\cJ(y \, | \, x) + (\Lambda f_m)(x)\right| \leq \varepsilon.
\end{equation} 
\end{enumerate}

Define for all $m,n$ the function $f_{m,n} = f_{m} \circ \eta_n \in C_b(X_n)$. By definition, we have for each $m$ that $\LIM f_{m,n} = f_m$, see Remark \ref{remark:approximating_functions}. Now choose $y_{n,\varepsilon} \in \cY_n$ such that 
\begin{align*}
\Lambda_n f_{m,n}(x_n) & = \sup_z f_{m,n}(z) - \cJ^n(z \, | \, x_n) \notag \\
& \leq f_{m,n}(y_{n,\varepsilon}) - \cJ^n(y_{n,\varepsilon} \, | \, x_n) + \varepsilon. 
\end{align*}
This implies
\begin{equation} \label{eqn:choice_of_ynm}
\cJ^n(y_{n,\varepsilon} \, | \, x_n) \leq f_{m,n}(y_{n,\varepsilon}) - \Lambda_nf_{m,n}(x_n) + \varepsilon.
\end{equation} 
An appropriate combination a diagonal argument, \eqref{eqn:choice_of_ynm}, \eqref{eqn:difference_V_and_I} and \eqref{eqn:upper_bound_on_fm_on_K}, $\LIM f_{m,n} = f_m$ and $\LIM \Lambda_n f_{m,n} = \Lambda f_m$. Will yield the final result. To make this argument rigorous, we need to establish that the sequences $y_{n,\varepsilon}$ are in appropriate compact sets and allow for an appropriate limit point $y$ along a diagonal.

\smallskip

Because $\LIM f_{m,n} = f_m$, we obtain by \eqref{eqn:convergence_of_loglaplace_in_conditionalupperbound} that $\LIM \Lambda_n f_{m,n} = \Lambda f_m$. As $x_n \in K_n^q$ and $\eta_n(x_n) \rightarrow x$, we can choose $N = N(\varepsilon,m(\varepsilon)) = N(\varepsilon)$ such that for $n \geq N$, we have (see also Remark \ref{remark:LIM_equivalence})
\begin{gather} 
\left|\Lambda_nf_{m,n}(x_n) - \Lambda f_m(x)\right| \leq \varepsilon, \label{eqn:difference_Lambda_n_f_m} \\
\sup_{y' \in K_n^{q'}} \left|f_{m,n}(y') - f_{m}(\eta_n(y')) \right| \leq \varepsilon. \label{eqn:difference_n_f_m}
\end{gather} 
We now estimate $\cJ^n(y_{n,\varepsilon} \, | \, x_n)$. This will show that $(x_n,y_n)$ are in the sets of \eqref{eqn:conditional_upper_bound_compact_set}, which allows us to carry out a diagonal argument. First, note that $f_{m,n} \leq M$ uniformly in $m$ and $n$. By \eqref{eqn:choice_of_ynm}, \eqref{eqn:difference_Lambda_n_f_m} and \eqref{eqn:difference_V_and_I} we find that
\begin{equation} \label{eqn:proof_cond_up_bound_first_found_on_constructed_seq}
\begin{aligned}
\cJ^n(y_{n,\varepsilon} \, | \, x_n) & \leq f_{m,n}(y_{n,\varepsilon}) - \Lambda_n f_{m,n}(x_n) + \varepsilon \\
& \leq  - \Lambda f_m(x) + 2\varepsilon + C   \\
&  \leq \cJ(y \, | \, x) + 3\varepsilon + C 
\end{aligned}
\end{equation}
for $n \leq N$. Because $\cI^n(x_n,y_{n,\varepsilon}) \leq c + \cJ^n(y_{n,\varepsilon} \, | \, x_n)$, this implies that $y_{n,\varepsilon} \in K_n^{q'}$. Note that because of this the bound in \eqref{eqn:proof_cond_up_bound_first_found_on_constructed_seq} can be improved by using $\sup_{z \in K} f_m(z) \leq 0$ obtained in \eqref{eqn:upper_bound_on_fm_on_K} instead of $\sup_{m} f_{m} \leq C$. Indeed by \eqref{eqn:difference_n_f_m}, we find for $n \geq N$ that we can replace $C$ by $\varepsilon$ in the chain of inequalities of \eqref{eqn:proof_cond_up_bound_first_found_on_constructed_seq} and obtain
\begin{equation} \label{eqn:proof_basic_conditional_limsup_bound}
\cJ^n(y_{n,\varepsilon} \, | \, x_n) \leq \cJ(y \, | \, x) + 4\varepsilon.
\end{equation}

We now extract a sequence $y_n$ by a diagonal argument.	 For $\varepsilon$ taking values in the sequence $\{1/k\}_{k \geq 4}$, we find as above constants $N(1/k)$ which we assume without loss of generality to satisfy $N(1/(k+1)) \geq N(1/k) + 1$ making sure that the sequence $k \mapsto N(1/k)$ diverges to infinity. Define $y_n := y_{n,\epsilon_n}$, where
\begin{equation} \label{eqn:conditional_upper_bound_choice_kn}
\epsilon_n := \begin{cases}
\frac{1}{4} & \text{if } n \leq N(1/4) \\
\frac{1}{k} & \text{if } N\left(\frac{1}{k} \right) < n \leq N\left(\frac{1}{k+1}\right), k \geq 4.
\end{cases}
\end{equation}

By \eqref{eqn:proof_basic_conditional_limsup_bound}
\begin{equation*}
\limsup_{n \rightarrow \infty} \cJ^n(y_n \, | \, x_n) \leq \cJ(y \, | \, x),
\end{equation*}
so that by \eqref{eqn:conditional_gamma_upperbound_I0}, we have
\begin{equation*}
\limsup_{n \rightarrow \infty}  \cI^n(x_n,y_n) \leq \cI(x,y).
\end{equation*}
Thus, we are left to prove that $y_n$ converges to $y$.

\smallskip

As $y_n \in K_n^q$ for all $n$, Assumption \ref{assumption:abstract_spaces_q} (b) implies that $\eta_n(y_n) \in K$.T Thus, every subsequence of $\eta_n(y_n)$ contains a further subsequence that converges to a limit in $K^{q'}$. Indeed let $\eta_{n'}(y_{n'})$ be such a converging subsequence and denote its limit by $\hat{y} \in K^{q'}$ as $n' \rightarrow \infty$. If for all these subsequences we have that $\hat{y} = y$ then it follows that $y_n$ converges to $y$ by general arguments.

\smallskip

Without loss of generality, we assume that the sequence $\eta_n(y_n)$ itself converges to $\hat{y}$. Fix $n \geq N(1/4)$. We prove that $\eta_n(y_n) \in B_{\varepsilon_n}(y)$. As $\varepsilon_n \rightarrow 0$ ans $k \mapsto N(1/k)$ diverges, this will establish that $\hat{y} = y$.

By \eqref{eqn:conditional_upper_bound_choice_kn}, we have that $n \geq N(\varepsilon_n)$. Thus, we are in a situation that we can apply all bounds from the first part of the proof. In the first part of the proof the choice of $m$ depended on an arbitrary $\varepsilon$ and was chosen such that $m(\varepsilon) \geq \varepsilon^{-1}$. We now make the choice of $m$ dependent on $\varepsilon_n$ and thus on $n$. Note, however, that in the inequalities below, we work for \textit{fixed} $n$. We obtain that
\begin{align*}
f_{m}(\eta_n(y_n)) & \geq f_{m,n}(y_{n}) - \varepsilon_n & & \text{by \eqref{eqn:difference_n_f_m}} \\
& = f_{m,n}(y_{n,\epsilon_n}) - \varepsilon_n & & \text{by \eqref{eqn:conditional_upper_bound_choice_kn}}  \\
& \geq \Lambda_n f_{m,n}(x_n) + \cJ^n(y_{n,\epsilon_n} \, | \, x_n) - 2 \epsilon_n & & \text{by \eqref{eqn:choice_of_ynm}} \\
& \geq \Lambda_n f_{m,n}(x_n) - 2 \epsilon_n & & \text{as $\cJ^n \geq 0$} \\
& \geq \Lambda f_m(x) - 3 \epsilon_n & & \text{by \eqref{eqn:difference_Lambda_n_f_m}} \\
& \geq - \cJ(y \, | \, x) - 4 \epsilon_n & & \text{by \eqref{eqn:difference_V_and_I}}
\end{align*}
As $4\varepsilon_n \leq 1$, we find by \eqref{eqn:function_fm_forces_in_small_set} that $\eta_n(y_n) \in B_{\varepsilon_n} \cap K$. As this holds for all $n$, we infer that $\hat{y} = y$.

\end{proof}

\begin{proof}[Proof of Lemma \ref{lemma:upper_bound_finite_dim}]
	Fix $y_0,\dots,y_k \in X$. We show that there are $y_{0,n},\dots,y_{k,n} \in X_n$ such that $\eta_n(y_{i,n}) \rightarrow y_i$ and 
	\begin{equation*}
	\limsup_n I^n[t_{0,n},\dots,t_{k,n}](y_{0,n},\dots,y_{k,n}) \geq I[t_0,\dots,t_k](y_0,\dots,y_k).
	\end{equation*}
	By Lemma \ref{lemma:contracting_equi_coercivity}, Proposition \ref{proposition:conditional_upper_bound} taking $\cX = X^j, \cY = X$ and $\Lambda_n f_n := V_n(t_{j+1,n}-t_{j,n})f_n$, and induction on the dimension $j$, we find $(y_{0,n},\dots,y_{k,n})$ such that $y_{i,n} \rightarrow y_i$ for all $i$ and
	\begin{equation*}
	\limsup_{n \rightarrow \infty} I^n[t_{0,n},\dots,t_{k,n}](y_{0,n},\dots,y_{k,n}) \leq I[t_0,\dots,t_k](y_0,\dots,y_k).
	\end{equation*}
\end{proof}

\begin{proof}[Proof of Proposition \ref{proposition:upper_bound}]
Fix $\gamma^* \in \Omega$. Without loss of generality assume that $I(\gamma^*) < \infty$. We construct a sequence $\gamma_n \in \Omega_n$ such that $\eta_n(\gamma_n) \rightarrow \gamma^*$ and $\limsup_n I^n(\gamma_n) \leq I(\gamma^*)$. Fix $\varepsilon > 0$. 

Set 
\begin{equation*}
K := \{\gamma^*\} \cup \text{closure} \bigcup_{n \geq 1} \left\{\eta_n(\gamma) \, | \, \gamma \in \Omega_n, \, I^n(\gamma) \leq I(\gamma^*) + 1 \right\}.
\end{equation*}
which is a compact set by the equi-coercivity of $\{I^n\}_{n \geq 1}$. Let $d$ be a metric on $K$, which exists by Assumption \ref{assumption:path_space} (a).

By Assumption \ref{assumption:path_space} (c), we find times $0 = t_0 < t_1 < \dots < t_k$ in $\Delta_{\gamma^*}^c$ and an open set $U' \subseteq \hat{X}^{k+1}$ containing $(\gamma^*(t_0),\dots,\gamma^*(t_k))$ such that
\begin{equation} \label{eqn:proof_upper_bound_reduction_to_finite_nr_of_coordinates}
\left\{y \in \Omega \, | \, (y(t_0),\dots,y(t_k)) \in U' \right\} \cap K \subseteq B_\varepsilon(\gamma^*)\cap K.
\end{equation}

The representation of $I$ in \eqref{eqn:functional_project2}, using that $(t_0,\dots,t_k) \in \Delta_{\gamma^*}^c$, yields
\begin{equation} \label{eqn:proof_upper_bound_estimate_of_I}
I(\gamma^*) \geq I[t_0,\dots,t_k](\gamma^*(t_0),\dots,\gamma^*(t_k)).
\end{equation}
By Lemma \ref{lemma:upper_bound_finite_dim}, taking $(t_{0,n},\dots,t_{k,n}) = (t_0,\dots,t_k)$ for each $n$, we find $(y_{0,n},y_{k,n}) \in X_n^k$ such that $\eta_n(y_{i,n}) \rightarrow y_i$ and 
\begin{equation}
\limsup_n I^n[t_0,\dots,t_k](y_{0,n},\dots,y_{k,n}) \leq I[t_0,\dots,t_k](\gamma^*(t_0),\dots,\gamma^*(t_k)). 
\end{equation}
Thus, there is some $N = N(\varepsilon)$ such that for $n \geq N$, we have
\begin{gather}
(\eta_n(y_{0,n}),\dots, \eta_n(y_{k,n})) \in U' \label{eqn:proof_upper_bound_openset} \\
I^n[t_0,\dots,t_k](y_{0,n},\dots,y_{k,n}) \leq I(\gamma^*) + \varepsilon \label{eqn:proof_upper_bound_finite_d_choice_bound_I}
\end{gather}

By \eqref{eqn:functional_contract2}, we can pick a curve $\gamma_{n,\varepsilon}$ in $K$ such that $\gamma_{n,\varepsilon}(t_i) = y_{n,i}$ and $t_i \in \Delta_{\gamma_{n,\varepsilon}}^c$ for all $i$ and
\begin{equation}  \label{eqn:upper_bound_choice_of_recovery_before_diagonal}
I^n(\gamma_{n,\varepsilon}) \leq  I^n[t_1,\dots,t_k](y_{n,0},\dots,y_{n,k}) + \frac{1}{n}
\end{equation}
We will construct the recovery sequence by using a diagonal argument. First note that by \eqref{eqn:proof_upper_bound_reduction_to_finite_nr_of_coordinates}, \eqref{eqn:upper_bound_choice_of_recovery_before_diagonal} and \eqref{eqn:proof_upper_bound_openset} we have for $n \geq N(\varepsilon)$ that
\begin{gather}
I^n(\gamma_{n,\varepsilon}) \leq I(\gamma^*) + \frac{1}{n} + \varepsilon, \label{eqn:proof_upper_bound_before_diagonal_bound} \\
\gamma_{n,\varepsilon} \in B_\varepsilon(\gamma^*). \label{eqn:proof_upper_bound_epsilon_close}
\end{gather}
We proceed by a diagonal argument. Without loss of generality, we can assume that $N(1/(m+1)) \geq N(1/m) + 1$. Define 
\begin{equation*}
m(n) := \begin{cases}
1 & \text{for } n \leq N(1) \\
m^{-1} & \text{for } N\left(\frac{1}{m} \right) < n \leq N\left(\frac{1}{m+1}\right), m \geq 1
\end{cases}
\end{equation*}
and $\gamma_n := \gamma_{n,m(n)}$. Note that as $m \mapsto N(1/m)$ diverges, we have that the constant $\lim_n m(n)^{-1} = 0$. Thus, we infer from \eqref{eqn:proof_upper_bound_epsilon_close} that $\eta_n(\gamma_n) \rightarrow \gamma^*$ and by \eqref{eqn:proof_upper_bound_before_diagonal_bound} that
\begin{equation*}
\limsup_{n \rightarrow \infty} I^n(\gamma_n) \leq I(\gamma^*).
\end{equation*}
\end{proof}

\subsection{Proofs of Theorem \ref{theorem:Gamma_convergence_via_semigroup_convergence} and \ref{theorem:Gamma_convergence}} \label{subsection:proof_main_result}

	\begin{proof}[Proof of Theorem \ref{theorem:Gamma_convergence_via_semigroup_convergence}]
		The lower bound follows from Proposition \ref{proposition:buc_cont_equiv_to_double_bound}. The assumptions of the theorem imply assumptions (i) and (a) of the proposition. The upper bound follows from Proposition \ref{proposition:upper_bound}.
	\end{proof}

	The following result is proven in Theorems 4.7 and 5.1 in \cite{Kr19}
	
	\begin{theorem} \label{theorem:CL_extend}
		Let Condition \ref{condition:convergence_of_generators_and_conditions_extended_supsuperlim} be satisfied. Suppose that 
		\begin{equation*}
		H_\dagger  \subseteq LSC_l(X) \times USC_u(Y),  \qquad H_\ddagger  \subseteq USC_u(X) \times LSC_l(Y).
		\end{equation*}
		are two operators such that $H_\dagger \subseteq ex-\subLIM H_{n,\dagger}$ and $H_\ddagger \subseteq ex-\superLIM H_{n,\ddagger}$
		
		\smallskip
		
		Let $D \subseteq C_b(X)$ be quasi-dense in $C_b(X)$. Suppose that for each $\lambda > 0$ and $h \in D$ the comparison principle holds for 
		\begin{equation} \label{eqn:HJ_semigroup_conv_result} 
		f - \lambda H_\dagger f = h, \qquad f - \lambda H_\ddagger f = h.
		\end{equation}
		
		Then there is a pseudo-resolvent $R(\lambda) : C_b(X) \rightarrow C_b(X)$ that is locally strictly equi-continuous on bounded sets with the same choice of $\hat{q}$ as in Condition \ref{condition:convergence_of_generators_and_conditions_extended_supsuperlim} \eqref{item:convH_strict_equicont_resolvents}. For each $h \in D$ and $\lambda >0$ the function $R(\lambda)h$ is the unique viscosity solution to \eqref{eqn:HJ_semigroup_conv_result}. Let $V(t)$ the semigroup generated by $R(\lambda)$ with generator
		\begin{equation} 
		\widehat{H} := \bigcup_{\lambda} \left\{\left(R(\lambda)h, \frac{R(\lambda) h - h}{\lambda}\right) \, \middle| \, h \in C_b(X) \right\}. \label{def:hatH}
		\end{equation}
		Denote by $\cD$ and the quasi-closure of the uniform closure of $\cD(\widehat{H})$.
		Then
		\begin{enumerate}[(a)]
			\item \label{item:semigroup_density_domains} We have $\widehat{H} \subseteq ex-\LIM \widehat{H}_n$ as in Definition \ref{definition:abstract_LIM}. That is, for all $(f,g) \in \widehat{H}$ there are $(f_n,g_n) \in \widehat{H}_n$ such that $\LIM f_n = f$ and $\LIM g_n = g$.
			\item \label{item:semigroup_extensionV} The semigroup $V(t)$ extends to the quasi-closure $\cD$ of $\cD(\widehat{H})$ on which it is locally strictly equi-continuous on bounded sets.
			\item  \label{item:semigroup_approx_of_domain} For each $f \in \cD$ there are $f_n$ in the uniform closures of $\cD(\widehat{H}_n)$ such that $\LIM f_n = f$.
			\item \label{item:semigroup_convergence_semigroups}  If $f_n$ are in the uniform closures of $\cD(\widehat{H}_n)$ and $f \in \cD$ such that $\LIM f_n =f$ and $t_n \rightarrow t$ then $\LIM V_n(t_n)f_n = V(t) f$. 
		\end{enumerate}
	\end{theorem}

\begin{proof}[Proof of Theorem \ref{theorem:Gamma_convergence}]
	The result will follow from Theorem \ref{theorem:Gamma_convergence_via_semigroup_convergence} if we can prove convergence of semigroups. This, however, follows immediately from Theorem \ref{theorem:CL_extend} \ref{item:semigroup_convergence_semigroups}. Finally, by Assumption (f) of Theorem \ref{theorem:Gamma_convergence}, we find $\cD = C_b(X)$ in Theorem \ref{theorem:CL_extend} so that the semigroup $v(t)$ is defined on the whole of $C_b(X)$.
\end{proof}

\appendix

\section{Dual functionals}

The proofs of the main results on Gamma convergence results will use convergence of dual functionals. These results are somewhat standard and are therefore included as an Appendix.

\begin{definition}[A dual functional]
	For $J : X \rightarrow [0,\infty]$, define for every measurable $f$ the functional $\Lambda(f) = \sup_x f(x) - J(x)$.
\end{definition}

By definition, it follows that $J(x) \geq \sup_{f \in M_b(X)} \left\{f(x) - \Lambda(f) \right\}$.

\begin{definition}
	We say that a collection $D \subseteq M_b(X)$ is \textit{functional determining} if for any coercive lower semi-continuous functional $J$ and dual functional $\Lambda$, we have
	\begin{equation*}
	J(x) = \sup_{f \in D} f(x) - \Lambda(f).
	\end{equation*}
	We say that $D$ is \textit{bounded from above} if $\sup_{f \in D} \sup_{x \in X} f(x) < \infty$. 
\end{definition}

We will consider the following class of functions that can be used to `isolate points' by a functions that are close to $0$ in a neighbourhood of the point and very small further away. These classes of functions will turn out to be functional determining.

\begin{definition} \label{definition:isolating_points}
	A collection of functions $D \subseteq C_b(X)$ is said to \textit{isolate points} if for all $x \in X$, constants $m > 0$, compact sets $K \subseteq X$ and open sets $U \subseteq X$ with $x \in U$ there is a function $f \in D$,  such that 
	\begin{enumerate}[(a)]
		\item $|f(x)| \leq m^{-1}$,
		\item $\sup_{y \in K} f(y) \leq 0$,
		\item $\sup_{y \in K \cap U^c} f(y) \leq - m$.
	\end{enumerate} 
	We say that $D$ is \textit{bounded above by $M \in \bR$} if $\sup_{f \in D} \sup_{x \in X} f(x) \leq M$. We say that $D$ is \textit{bounded above} if there is some $M \in \bR$ such that $D$ is bounded above by $M$.
\end{definition}

\begin{lemma} \label{lemma:isolating_points_gives_functions_that_approximate}
	Let $X$ be some space with metrizable compact subsets. Let $J$ be a coercive lower semi-continuous functional on $X$. Suppose $D \subseteq M_b(X)$ is bounded from above and isolates points. 
	Let $x \in X$ be a point such that $J(x) < \infty$ and let $K$ be a compact set. Denote by $d_K$ a metric on $K$ and denote by $B_r(y) := \{z \in K \, | \, d_K(y,z) < r\}$.
	
	Then, there is a metric $d_K$ a metric on $K$ (denote by $B_r(y) := \{z \in K \, | \, d_K(y,z) < r\}$) and functions $f_m \in D$ such that
	\begin{enumerate}[(a)]
		\item $|f_m(x)| \leq m^{-1}$
		\item $\lim_{m \rightarrow \infty} \sup_{y \in K} f_m(y) \leq 0$
		\item 
		\begin{equation*}
		\sup_{y\in B_{1/m}^c(x) \cap K} f_m(y) \leq - m.
		\end{equation*}
	\end{enumerate}
	and $\lim_{m \rightarrow \infty} \Lambda(f_m) = -J(x)$
\end{lemma}

\begin{proof}
	Suppose that $J(x) = M_1 < \infty$ and suppose that $M_2 \geq 0$ is a global upper bound for the set $D$.  Let $K_1 := K \cup \{ y \in X \, | \, J(y) \leq M_1 + M_2 + 2\}$. Let $d$ be some metric on $K_1$ and denote by $d_K$ the restriction of $d$ to $K$.
	
	For $m \geq 1$ let $f_m$ be a function satisfying (a)-(c) of Definition \ref{definition:isolating_points} with compact set $K_1$, base-point $x$, and $U$ any open set in $X$ such that $U \cap K_1 =  B_{1/m}(x)$.
	
	First of all,
	\begin{equation} \label{eqn:isolating_points_basic_inequality}
	\Lambda(f_m) \geq f_m(x) - J(x) \geq - \frac{1}{m} - J(x).
	\end{equation}
	
	Second of all, let $x_m$ be a point such that 
	\begin{equation}\label{eqn:isolating_points_inequality_varying_points}
	\Lambda(f_m) \leq f_m(x_m) - J(x_m) + \frac{1}{m}.
	\end{equation}
	
	We conclude that
	\begin{equation}
	J(x_m) \leq f_m(x_m) - \Lambda(f_m) + \frac{1}{m} \leq M_2 + J(x) + \frac{2}{m},
	\end{equation}
	which implies that $x_m \in K_1$.
	
	\smallskip
	
	Additionally, by \eqref{eqn:isolating_points_basic_inequality} and \eqref{eqn:isolating_points_inequality_varying_points}, we have
	\begin{align*}
	f_m(x_m) & \geq \Lambda(f_m) + J(x_m) - \frac{1}{m} \\
	& \geq f_m(x) - J(x) + J(x_m) - \frac{1}{m}. 
	\end{align*}
	Using that $f_m(x) \geq - \frac{1}{m}$ and $J(x_m) \geq 0$, we find
	\begin{equation*}
	f_m(x_m) \geq -M_1 - \frac{2}{m} \geq - M_1 - 2.
	\end{equation*}
	By property (c) of Definition \ref{definition:isolating_points}, combined with the fact that $x_m \in K_1$, we find that if $m \geq M_1 + 2$, then $x_m \in B_{1/m}(x)$. We conclude that $\lim_{m \rightarrow \infty} x_m = x$.
	
	\smallskip
	
	By the lower semi-continuity of $J$ we find that $\liminf_{m \rightarrow \infty} J(x_m) \geq J(x)$. On the other hand,
	\begin{align*}
	J(x) & \geq f_m(x) - \Lambda(f_m) \\
	& \geq f_m(x) - f_m(x_m) + J(x_m) - \frac{1}{m} \\
	& \geq J(x_m) - \frac{2}{m}
	\end{align*} 
	by our choice of $f_m$. This yields that $\limsup_{m \rightarrow \infty} J(x_m) \leq J(x)$, which implies that 
	\begin{equation} \label{eqn:convergence_of_J}
	\lim_m J(x_m) = J(x).
	\end{equation}
	
	We will now prove that $J(x) = - \lim_m \Lambda(f_m)$. By \eqref{eqn:isolating_points_basic_inequality}, we find $J(x) \geq \limsup_m - \Lambda(f_m)$, whereas \eqref{eqn:isolating_points_inequality_varying_points}, combined with the fact that $x_m \in K$ implies $f_m(x_m) \leq 0$, yields
	\begin{equation*}
	J(x_m) \leq - \frac{1}{m} - \Lambda(f_m).
	\end{equation*}
	By \eqref{eqn:convergence_of_J} the limit of the left hand side exists and equals $J(x)$, which gives $J(x) \leq \liminf_m - \Lambda(f_m)$. This proves the claim.
	
\end{proof}

The following result is the analogue of Proposition 3.20 in \cite{FK06}.

\begin{proposition} \label{proposition:bounded_above_and_isolating_points_implies_functional_determining}
	Let $D \subseteq M_b(X)$ contain a set $D_0$ that is bounded above and isolating points. Then $D$ is functional determining.
\end{proposition}

Note that $C_b(X)$ always contains a set which isolates points and is bounded above.

\begin{proof}[Proof of Proposition \ref{proposition:bounded_above_and_isolating_points_implies_functional_determining}]
	Let $J$ be a coercive lower semi-continuous functional on $X$ and let $\Lambda$ be its dual functional. Using the definition of $\Lambda$, we find
	\begin{equation*}
	\sup_{f \in D}  f(x) - \Lambda(f) = \sup_{f \in D}  \inf_{y \in X} f(x) - f(y) + J(y) \leq J(x)
	\end{equation*}	
	by taking $y= x$. The converse inequality in the case that $J(x) < \infty$ follows by Lemma \ref{lemma:isolating_points_gives_functions_that_approximate}. 
	
	\smallskip
	
	Suppose now that $J(x) = \infty$. Let $M_2$ be the uniform upper bound on the functions in $D_0$. Fix some constant $m \geq 1$ and set $K = \{y \, | \, J(y) \leq m\}$. By the lower semi-continuity of $J$, there is some open set $U \subseteq X$ such that $x \in U$ and if $y \in U$, then $J(y) \geq m$. Choose $f_m \in D_0$ as in Definition \ref{definition:isolating_points} with point $x$, constant $m$, compact set $K$ and open set $U$. Then, we have the following three results.
	\begin{enumerate}[(a)]
		\item By the uniform bound on $f \in D_0$ and the bound for $J$ on $U$, we have
		\begin{equation*}
		\sup_{ y \in U} f_m(y) - J(y) \leq M_2 - m.
		\end{equation*}
		\item By choice of our function $f_m$ and the fact that $J(x) \geq 0$, we find
		\begin{equation*}
		\sup_{ y \in U^c(x) \cap K} f_m(y) - J(y) \leq - m \leq  M_2 - m.
		\end{equation*}
		\item By our choice of the set $K$ and the uniform bound on $f \in D_0$, we have
		\begin{equation*}
		\sup_{ y \in U^c(x)\cap K^c} f_m(y) - J(y) \leq M_2 - m.
		\end{equation*}
	\end{enumerate}
	We conclude from (a),(b) and (c) that
	\begin{equation*}
	\Lambda(f_m) = \sup_{ y } f_m(y) - J(y) \leq M_2 - m.
	\end{equation*}
	Therefore,
	\begin{align*}
	I(x) & \geq \limsup_{m \rightarrow \infty} f_m(x) - \Lambda(f_m) \\
	& \geq \limsup_{m \rightarrow \infty} - m^{-1} + m - M_2 \\
	& = \infty.
	\end{align*}
\end{proof}

\section{Viscosity solutions, operators, the strict topology and the convergence of spaces} \label{appendix:operators_topology_and_viscosity_solutions}

We repeat some of the basic notions of \cite{Kr19} to support the generator approach to the convergence of semigroups. All proofs can be found therein.

\subsection{Viscosity solutions}

Let $X$ and $Y$ be two spaces. Let $\gamma : Y \rightarrow X$ be continuous and surjective.

We consider operators $A \subseteq M(X) \times C(Y)$. If $A$ is single valued and $(f,g) \in A$, we write $Af := g$. We denote $\cD(A)$ for the domain of $A$ and $\cR(A)$ for the range of $A$.

\begin{definition}
	Let $A_\dagger \subseteq LSC_l(X) \times USC_u(Y)$ and $A_\ddagger \subseteq USC_u(X) \times LSC_l(Y)$. Fix $h_1,h_2 \in M(X)$. Consider the equations
	\begin{align} 
	f -  A_\dagger f & = h_1, \label{eqn:differential_equation_Adagger} \\
	f - A_\ddagger f & = h_2. \label{eqn:differential_equation_Addagger}
	\end{align}
	\begin{description}
		\item[Classical solutions] We say that $u$ is a \textit{classical subsolution} of equation \eqref{eqn:differential_equation_Adagger} if there is a $g$ such that $(u,g) \in A_\dagger$ and $u - g \leq h$. We say that $v$ is a \textit{classical supersolution} of equation \eqref{eqn:differential_equation_Addagger} if there is $g$ such that $(v,g) \in A_\ddagger$ and $v - g \geq h$. We say that $u$ is a \textit{classical solution} if it is both a sub- and a supersolution.
		\item[Viscosity subsolutions] We say that $u : X \rightarrow \bR$ is a \textit{subsolution} of equation \eqref{eqn:differential_equation_Adagger} if $u \in USC_u(X)$ and if, for all $(f,g) \in A_\dagger$ such that $\sup_x u(x) - f(x) < \infty$ there is a sequence $y_n \in Y$ such that
		\begin{equation} \label{eqn:subsol_optimizing_sequence}
		\lim_{n \rightarrow \infty} u(\gamma(y_n)) - f(\gamma(y_n))  = \sup_x u(x) - f(x),
		\end{equation}
		and
		\begin{equation} \label{eqn:subsol_sequence_outcome}
		\limsup_{n \rightarrow \infty} u(\gamma(y_n)) - g(y_n) - h_1(\gamma(y_n)) \leq 0.
		\end{equation}
		\item[Viscosity supersolution] 	We say that $v : X \rightarrow \bR$ is a \textit{supersolution} of equation \eqref{eqn:differential_equation_Addagger} if $v \in LSC_l(X)$ and if, for all $(f,g) \in A_\ddagger$ such that $\inf_x v(x) - f(x) > - \infty$ there is a sequence $y_n \in Y$ such that
		\begin{equation} \label{eqn:supersol_optimizing_sequence}
		\lim_{n \rightarrow \infty} v(\gamma(y_n)) - f(\gamma(y_n))  = \inf_x v(x) - f(x),
		\end{equation}
		and
		\begin{equation} \label{eqn:supersol_sequence_outcome}
		\liminf_{n \rightarrow \infty} v(\gamma(y_n)) - g(y_n) - h_2(\gamma(y_n)) \geq 0.
		\end{equation}
		\item[Viscosity solution] We say that $u$ is a \textit{solution} of the pair of equations \eqref{eqn:differential_equation_Adagger} and \eqref{eqn:differential_equation_Addagger} if it is both a subsolution for $A_\dagger$ and a supersolution for $A_\ddagger$.
		\item[Comparison principle] We say that  \eqref{eqn:differential_equation_Adagger} and \eqref{eqn:differential_equation_Addagger} satisfy the \textit{comparison principle} if for every subsolution $u$ to \eqref{eqn:differential_equation_Adagger} and supersolution $v$ to \eqref{eqn:differential_equation_Addagger}, we have
		\begin{equation} \label{eqn:comparison_estimate}
		\sup_x u(x) - v(x) \leq \sup_x h_1(x) - h_2(x).
		\end{equation}
		If $H = A_\dagger = A_\ddagger$, we will say that the comparison principle holds for $f - \lambda Af = h$, if for any subsolution $u$ for $f - \lambda Af = h_1$ and supersolution $v$ of $f - \lambda Af = h_2$ the estimate in \eqref{eqn:comparison_estimate} holds. 
	\end{description}
\end{definition}

Usually, $Y = X$ and $\gamma(x) = x$ simplifying the definitions above. $Y$ can be chosen distinct from $X$ for example in the setting that the operator $A$ is obtained as a limit from operators $A_n$ for which there is a natural separation of time-scales.

\subsection{Operators and the strict and buc topology}

In addition to normed spaces, we consider bounded and uniform convergence on compacts (buc-convergence). This notion of convergence for functions on $C_b(X)$ is more natural from an applications point of view. This is due to the fact that it is the restriction of the locally convex strict topology restricted to sequences, see e.g. \cite{Bu58,Se72}. Indeed, it is the strict topology for which most well known results generalize (under appropriate conditions on the topology, e.g. $X$ Polish): Stone-Weierstrass, Arzelà-Ascoli and the Riesz representation theorem. We define both notions.

\begin{definition}[buc convergence] \label{definition:buc_convergence}
	Let $f_n \in C_b(X)$ and $f \in C_b(X)$. We say that $f_n$ converges \textit{bounded and uniformly on compacts} (buc) if $\sup_n \vn{f_n} < \infty$ and if for all compact $K \subseteq X$:
	\begin{equation} \label{eqn:buc_uniform}
	\lim_n \sup_{x \in K} \left|f_n(x) - f(x)\right| = 0.
	\end{equation}
\end{definition}

Note \eqref{eqn:buc_uniform} can be replaced by $f_n(x_n) \rightarrow f(x)$ for all sequences $x_n \in K$ that converge to $x \in K$.

\begin{definition}
	The \textit{(sub) strict topology} $\beta$ on the space $C_b(X)$ for a completely regular space $X$ is generated by the collection of semi-norms
	\begin{equation*}
	p(f) := \sup_n a_n \sup_{x \in K_n} |f(x)|
	\end{equation*}
	where $K_n$ are compact sets in $X$ and where $a_n \geq 0$ and $a_n \rightarrow 0$.
\end{definition}

\begin{remark} \label{remark:strict_implies_buc}
	The (sub)strict topology is the finest locally convex topology that coincides with the compact open topology on bounded sets. Thus, a sequence converges strictly if and only if it converges buc.
	
	In the literature on locally convex spaces, the strict topology is usually referred to as the substrict topology, but on Polish spaces, amongst others, these topologies coincide, see \cite{Se72}.
\end{remark}

\begin{definition}
	\begin{enumerate}[(a)]
		\item Denote $B_r := \left\{f \in C_b(X) \, \middle| \, \vn{f} \leq r \right\}$. We say that a set $D$ is \textit{quasi-closed} if for all $r \geq 0$ the set $D \cap B_r$ is closed for the strict topology (or equivalently for the compact open or buc topologies).
		\item We say that $\widehat{D}$ is the \textit{quasi-closure} of $D$ if $\widehat{D} = \bigcup_{r > 0} \widehat{D}_r$, where $\widehat{D}_r$ is the strict closure of $D \cap B_r$.
		\item We say that $D_1$ is \textit{quasi-dense} in $D_2$ if $D_1 \cap B_r$ is strictly dense in $D_2 \cap B_r$ for all $r \geq 0$.
	\end{enumerate}
\end{definition}

Next, we consider operators with respect to a hierarchy of statements regarding continuity involving the strict topology. 

\begin{proposition} \label{proposition:buc_cont_equiv_to_double_bound}
	Let $T : C_b(X) \rightarrow C_b(X)$. Consider
	\begin{enumerate}[(a)]
		\item $T$ is strictly continuous.
		\item \label{item:continuity_strict_intermediate} For all $\delta > 0$, $r > 0$, and compact sets $K$ there are $C_0(r)$, $C_1(\delta,r)$ and a compact set $\hat{K}(K,\delta,r)$ such that 
		\begin{equation*}
		\sup_{x \in K} |Tf(x) - Tg(x)| \leq \delta C_0(r) + C_1(\delta, r) \sup_{x \in \hat{K}(K,\delta,r)} |f(x) - g(x)|
		\end{equation*}
		for all $f,g \in C_b(X)$ such that $\vn{f} \vee \vn{g} \leq r$.
		\item $T$ is strictly continuous on bounded sets.
	\end{enumerate}
	Then (a) implies (b) and (b) implies (c).
\end{proposition}

\begin{remark}
	There is not much room between properties (a) and (c). In the case that $X$ is Polish space, and $T$ is linear then (a) and (c) are equivalent, see e.g. \cite[Corollary 3.2 and Theorem 9.1]{Se72}. It is unclear to the author whether (b) and (c) are equivalent in general.
\end{remark}

At various points in the paper, we will work with operators that are constructed by taking closures on dense sets. To do so, we need continuity properties. Even though working with (a) of \ref{proposition:buc_cont_equiv_to_double_bound} would be the desirable from a functional analytic point of view,  \eqref{item:continuity_strict_intermediate} is much more explicit, and also suffices for our analysis.

The following result is proven in \cite[Lemma A.11]{FK06}.

\begin{lemma} \label{lemma:extension_operator}
	Suppose that an operator $T : D \subseteq C_b(X) \rightarrow C_b(X)$ satisfies (b) of Proposition \ref{proposition:buc_cont_equiv_to_double_bound}.  Then $T$ has an extension to the quasi-closure $\widehat{D}$ of $D$ that also satisfies property \eqref{item:continuity_strict_intermediate} of Proposition \ref{proposition:buc_cont_equiv_to_double_bound}  (with the same choice of $\hat{K}$).
\end{lemma}

\subsection{Operators}

For an operator $A \subseteq M(X) \times M(Y)$ and $c \geq 0$ we write $cA \subseteq M(X) \times M(Y)$ for the operator
\begin{equation*}
c \cdot A := \left\{ (f,c\cdot g) \, \middle| \, (f,g) \in A \right\}.
\end{equation*}
Here we write $c \cdot g$ for the function
\begin{equation*}
c \cdot g(x) := \begin{cases}
cg(x) & \text{if } g(x) \in \bR, \\
\infty & \text{if } g(x) = \infty, \\
- \infty & \text{if } g(x) = - \infty.
\end{cases}
\end{equation*}

The next set of properties is mainly relevant in the setting that $Y = X$.

\begin{definition}
	
	\begin{description}
		\item[Contractivity] We say that $T \subseteq M(X) \times M(X)$ is \textit{contractive} if for all $f_1,f_2 \in \cD(T)$:
		\begin{align*}
		\sup_x Tf_1(x) - Tf_2(x) \leq \sup_x f_1(x) - f_2(x), \\
		\inf_x Tf_1(x) - Tf_2(x) \geq \inf_x f_1(x) - f_2(x).
		\end{align*}
		If in addition $T0 = 0$, contractivity implies that $\sup_x Tf(x) \leq \sup_x f(x)$ and $\inf_x Tf(x) \geq \inf_x f(x)$.
		\item[Dissipativity] We say $A \subseteq M(X) \times M(X)$ is \textit{dissipative} if for all $(f_1,g_1),(f_2,g_2) \in A$ and $\lambda > 0$ we have
		\begin{equation*}
		\vn{f_1 - \lambda g_1 - (f_2 - \lambda g_2)} \geq \vn{f_1 - f_2};
		\end{equation*}
		\item[The range condition] We say $A \subseteq M(X) \times M(X)$ satisfies \textit{the range condition} if for all $\lambda > 0$ we have: the uniform closure of $\cD(A)$ is a subset of $\cR(\bONE - \lambda A)$.
	\end{description}

\end{definition}

The following theorem was proven in  \cite{CL71} for accretive operators but can be easily translated into dissipative operators by changing $A$ by $-A$.

\begin{theorem}[Crandall-Liggett \cite{CL71}] \label{theorem:CL}
	Let $A$ be an operator on a Banach space $E$. Suppose that 
	\begin{enumerate}[(a)]
		\item $A$ is dissipative,
		\item $A$ satisfies the range condition.
	\end{enumerate}
	Denote by $R(\lambda,A) = \left(\bONE- \lambda A\right)^{-1}$. Then there is a strongly continuous (for the norm) contraction semigroup $S(t)$ defined on the uniform closure of $\cD(A)$ and for all $t \geq 0$ and $f$ in the uniform closure of $\cD(A)$
	\begin{equation*}
	\lim_n \vn{R\left(\tfrac{t}{n},A\right)^n f - S(t) f} = 0.
	\end{equation*}
\end{theorem}

In the context of this theorem, we say that the $A$ generates the semigroup $V(t)$.

\subsection{Semigroups, resolvents and generators on the space of continuous functions}

Finally, we consider the Crandall-Liggett theorem in the context of $C_b(X)$ and $M(X)$. We have seen that the natural topology is the strict topology, rather than the supremum topology.

\begin{definition}[Pseudo-resolvents]
	Consider a space $X$ and a subset $B$ such that $C_b(X) \subseteq B \subseteq M(X)$ on which we have a family of operators $R(\lambda) : B \rightarrow B$, for $\lambda > 0$.  We say that this family is a \textit{pseudo-resolvent} if $R(\lambda) 0 = 0$ for $\lambda > 0$ and if for all $\alpha < \beta$ we have
	\begin{equation*}
	R(\beta) = R(\alpha)\left(R(\beta) - \alpha \frac{R(\beta) - \bONE}{\beta} \right).
	\end{equation*}
\end{definition}

The next result is Proposition 3.10 in \cite{Kr19}.

\begin{proposition} \label{proposition:pseudo-resolvent_generation}
	Consider a space $X$ and a subset $B$ such that $C_b(X) \subseteq B \subseteq M(X)$. Let $R(\lambda)$ be a contractive pseudo-resolvent on $B$. Define
	\begin{equation*}
	\widehat{H} := \bigcup_{\lambda} \left\{\left(R(\lambda)h, \frac{R(\lambda) h - h}{\lambda}\right) \, \middle| \, h \in C_b(Z) \right\}. 
	\end{equation*}
	Then $\widehat{H}$ generates a semigroup $V(t)$ as in the Crandall-Liggett theorem, Theorem \ref{theorem:CL}: for all $t \geq 0$ and $f$ in the uniform closure of $\cD(\widehat{H})$
	\begin{equation*}
	\lim_n \vn{R\left(\tfrac{t}{n}\right)^n f - V(t) f} = 0.
	\end{equation*}
	
\end{proposition}

\begin{lemma} \label{lemma:pseudo-resolvent_generation_domain}
	Consider the setting of Proposition \ref{proposition:pseudo-resolvent_generation}. Suppose that 
	\begin{enumerate}[(a)]
		\item for all $h \in C_b(X)$ and $\lambda > 0$ we have $R(\lambda)h$ converges strictly to $h$.
		\item The semigroup $V(t)$ is locally strictly equi-continuous on bounded sets: for all compact sets $K \subseteq X$, $\delta > 0$ and $t_0 > 0$, there is a compact set $\widehat{K} = \widehat{K}(K,\delta,t_0)$ such that for all $n$ and $h_{1},h_{2} \in \cD(\widehat{H})$ and $0 \leq t \leq t_0$ that
		\begin{equation*}
		\sup_{y \in K} \left\{ V_n(t)h_{1}(y) - V_n(t)h_{2}(y) \right\} \leq \delta \sup_{x \in X} \left\{ h_{1}(x) - h_{2}(x) \right\} + \sup_{y \in \widehat{K}} \left\{ h_{1}(y) - h_{2}(y) \right\}.
		\end{equation*}
	\end{enumerate}
	Then $V(t)$ extends to a semigroup $V(t)$ on $C_b(X)$ that is locally strictly equi-continuous on bounded sets semigroup.
\end{lemma}

\begin{proof}
	By (a) the domain $\cD(\widehat{H})$ is strictly dense in $C_b(X)$. Thus by (b) and Lemma \ref{lemma:extension_operator} all operators $V(t)$ obtained via the Crandall-Liggett theorem extend to $C_b(X)$.
\end{proof}

In the context of Proposition \ref{proposition:pseudo-resolvent_generation} and Lemma \ref{lemma:pseudo-resolvent_generation_domain} we will say that the resolvent $R(\lambda)$ generates the semigroup $V(t)$.

\smallskip

\textbf{Acknowledgement}
Thanks go to Upanshu Sharma and Michiel Renger for drawing attention to the connection between large deviation theory and Gamma convergence.


\bibliographystyle{plain} 
\bibliography{../KraaijBib}{}

\end{document}